\documentclass[11pt,reqno]{amsart}
\usepackage{type1ec}
\usepackage[utf8]{inputenc}
\usepackage[T1]{fontenc}
\usepackage{comment}
\usepackage{amssymb,amsfonts,amsthm}

\usepackage{enumitem}
\usepackage{amsthm}
\usepackage{amssymb}
\usepackage[french,english]{babel}
\usepackage{enumerate}
\usepackage{setspace}

\usepackage{hyperref}

\usepackage{color}
\definecolor{Red}{cmyk}{0,1,1,0.2}

\usepackage{amsfonts,amsmath,layout}

\newcommand{\Q}{\mathbb Q}

\newcommand{\R}{\mathbb R}

\def\R{\mathbb R}

\def\E{\mathbb E}
\def\P{\mathbb P}

\def\ind{\mathbf{1}_}

\newcommand{\be}{\begin{equation}}
\newcommand{\ee}{\end{equation}}

\def\1{{\bf 1}}
\def\id{{\rm id}}

\def\ds{\displaystyle}

\newcommand{\pare}[1]{\left (#1\right )}
\newcommand{\croc}[1]{\left [#1\right ]}

\newtheorem{Theorem}{Theorem}[section]
\newtheorem{Definition}[Theorem]{Definition}
\newtheorem{Proposition}[Theorem]{Proposition}
\newtheorem{Lemma}[Theorem]{Lemma}
\newtheorem{Corollary}[Theorem]{Corollary}
\newtheorem{Remark}[Theorem]{Remark}

\newtheorem{Sketch of proof}[Theorem]{Sketch of proof}
\begin{document}
\title[Spider diffusion with spinning measure selected from its own local time]{On spider diffusions having a spinning measure selected from their own local time}
\author[Miguel Martinez \& Isaac Ohavi]{Miguel Martinez$^{\dagger}$, Isaac Ohavi$^{\star}$\\
{$^\dagger$ Universit\'e Gustave Eiffel, LAMA UMR 8050, France}\\{$^\star$ Hebrew University of
Jerusalem, Department of Mathematics and Statistics, Israël}}
\email{miguel.martinez@univ-eiffel.fr}
\email{isaac.ohavi@mail.huji.ac.il \& isaac.ohavi@gmail.com}
\thanks{Research partially granted by the Labex Bézout\\Research partially supported by the GIF grant 1489-304.6/2019}
\dedicatory{Version: \today}
\maketitle
\begin{abstract}
The aim of this article is to give several results related to Walsh's spider diffusions living on a star-shaped network that have a spinning measure selected from the own local time of the motion at the vertex (cf.\cite{Martinez-Ohavi Walsh}). We prove the corresponding It\^o's formula and give some global trajectory properties such as $L^1$-approximation of the local time and the Markov property.
Regarding the behavior of the process at the vertex, we show that that the distribution of the process is non atomic at the junction point and we characterize the instantaneous scattering distribution along some ray with the aid of the probability coefficients of diffraction. We obtain also a Feynmann-Kac representation for linear parabolic systems posed on star-shaped networks that where introduced in \cite{Martinez-Ohavi EDP} possessing a so-called {\it local-time Kirchhoff's boundary condition}.
\end{abstract}

\section{Introduction}\label{sec intro compa}

Walsh spider diffusion processes are currently being thoroughly studied and extended to various settings. Let us mention the following recent articles amongst the vast literature on the subject: in \cite{Ichiba} the authors propose the construction of stochastic integral equations related to Walsh semimartingales, in \cite{Ichiba-2} the authors compute the possible stationary distributions, in \cite{Karatzas-Yan} the authors investigate stopping control problems involving Walsh semimartingales, in \cite{Atar} the authors study related queuing networks, whereas \cite{Bayraktar} addresses the problem of finding related stopping distributions. We refer also to the introduction of \cite{Lejay} for a comprehensive survey on Skew Brownian motion, the reader may also find therein many older references on the subject.
\smallskip
Although difficult, several constructions of Walsh's diffusions have been proposed in the literature, see  for e.g. \cite{Barlow-Pitman-Yor} for a construction based on Feller’s semigroup theory, \cite{Salisbury} for a construction using the excursion theory for right processes, and also the very recent preprint \cite{Bayraktar-2} that proposes a new construction of Walsh diffusions using time changes of multi-parameter processes. 

Recall once again that in all these constructions, the spinning measure of the process -- that is strongly related somehow to {\it 'the way of selecting infinitesimally the different branches from the junction vertex'} -- remains constant through time. 

\medskip
\medskip

In our previous contribution \cite{Martinez-Ohavi Walsh} that should be regarded as the companion paper of this article, we have proved existence and uniqueness in the weak sense of a Walsh's spider process whose spinning measure and coefficients are allowed to depend on the local time at the junction vertex. 

Let us first briefly recall the methodology and the main lines that lead to the principal result obtained in \cite{Martinez-Ohavi Walsh} (Theorem 3.1).

Given $I$ a positive integer ($I\geq 2$), we define the star-shaped network $\mathcal{J}$ as:
\begin{align*}
\mathcal{J}~~:=~~\{{\bf 0}\}\cup \pare{(0,\infty)\times [I]}, ~~\text{with}~~[I] := \{1,\dots, I\},
\end{align*} 
where ${\bf 0} = \{(0,j), j\in [I]\}$ is the junction vertex equivalence class. We are given also $I$ pairs $(\sigma_{i},b_{i})_{i\in [I]}$ of mild coefficients from $[0,+\infty)$ to $\R$ satisfying the following condition of ellipticity: $\forall i\in [I],~\sigma_i>0$. Finally let $\big(\alpha_1,\ldots,\alpha_I)$ positive constants satisfying $\displaystyle \sum_{i=1}^I \alpha_i=1$, corresponding to the probability coefficients of diffraction of the spider along some ray. It was proved in \cite{Freidlin-Wentzell-2} that there exists a continuous Feller Markov process $\big(x(\cdot),i(\cdot)\big)$ valued in ${\mathcal{J}}$, such that the process $\big(x(\cdot)\big)$ satisfies the following stochastic differential equality:
\begin{equation*}
dx(t)= b_{i(t)}(x(t))dt + \sigma_{i(t)}(x(t))dW(t)+d\ell(t) \;,~~0\leq t\leq T.
\end{equation*}
In the above equality, the process $\ell(\cdot)$ is the local time of the process $\big(x(\cdot)\big)$ at the vertex $\bf 0$. Finally recall that the following It\^{o}'s formula was also proved in \cite{freidlinS}: 
\begin{eqnarray}\label{Ito Sheu}
\nonumber&\displaystyle df_{i(t)}(x(t))~~=~~ \Big(b_{i(t)}(x(t))\partial_xf_{i(t)}(x(t))+ \frac{1}{2}\sigma_{i(t)}^2(x(t))\partial_{x}^2f_{i(t)}(x(t))\Big)dt+\\
&
\displaystyle \partial_xf_{i(t)}(x(t))\sigma_{i(t)}(x(t))dW(t) 
+ \sum_{i=1}^{I}\alpha_{i}\partial_xf_i(0)d\ell(t),~~\P-\text{a.s,}
\end{eqnarray}
for any sufficiently regular $f$.

In the companion paper \cite{Martinez-Ohavi Walsh} of this contribution, the objective was to extend the above mentioned existence and uniqueness results obtained for spider motions in \cite{Freidlin-Wentzell-2}, by allowing now all coefficients of the process -- including the spinning measure -- to depend both on the own local time of the process spent at the junction together with the current running time. Therein, we took naturally the results stated in \cite{Freidlin-Wentzell-2} as our starting building block and we constructed 'by hand' a solution of a martingale problem that was purposely designed to take the presence of the local time in all the leading coefficients into account. More precisely, in \cite{Martinez-Ohavi Walsh} we built the spider process using a careful adaptation of the seminal construction for solutions of classical martingale problems that have $\R^d$ as the underlying state space, combining concatenation of probability measures and a tension argument (for more details regarding the choice of this methodology, we refer the reader to the Introduction in \cite{Martinez-Ohavi Walsh}). For the uniqueness part, the proof was achieved using a PDE argument that relates to the advances contained in \cite{Martinez-Ohavi EDP}, which deals with the well-posedness of parabolic systems posed on graphs, having a so-called {\it local-time Kirchhoff's boundary condition} designed for our purposes. Up to our knowledge \cite{Martinez-Ohavi Walsh} is the first result for the existence of a Walsh spider process possessing a non-constant spinning measure.

Before detailing the different main lines and results of this contribution, let us briefly explain why we believe that the addition of local time in the diffraction coefficients is both stimulating from a theoretical and practical point of view (see also the Introductions (sub sections: general motivations) in \cite{Martinez-Ohavi Walsh} and \cite{Ohavivisco}).

From a theoretical point of view, it appears first that the dependency of the local time in the coefficients of diffraction will lead to some novelty in the field of stochastic scattering control theory. It is expected that this dependency will allow to better understand how the diffraction of the spider is acting, especially regarding the behavior of the second order terms near $\bf 0$. In a related context of non linear PDE, I.Ohavi managed very recently in \cite{Ohavivisco} to obtain a comparison theorem (thus uniqueness) for continuous viscosity solution to some kind of Walsh’s spider Hamilton-Jacobi-Bellman system that possesses a new type of boundary condition at the vertex $\bf 0$ involving a non linear local time Kirchhoff ’s transmission (see \cite{Ohavivisco}). Let us emphasize that in \cite{Ohavivisco} the introduction of an external deterministic ’local-time’ variable $l$ -- that is the counterpart to the local time $\ell$ -- is one of the crucial ingredients to obtain the comparison principle. Note that even without the presence of the external variable $l$ in the original HJB problem, the 'artificial' introduction of this external variable in the problem allowed to extend the main results contained in \cite{Lions Souganidis 1} and \cite{Lions Souganidis 2} to the fully non linear and non degenerate framework.

From a practical point of view, let us imagine for instance a punctual source of light that crosses a plane at some point $O$. In this case, one can imagine that we constrain a Brownian particle to move along a finite number of rays, with different magnetic and electronic properties, that are joined on a 'spider web' whose central vertex lies at $\bf 0$.  When passing at the vertex junction, the Brownian particle gets directly hit by the punctual source light and this modifies its electronic properties. Thus, the particle is instantaneously attracted in a more privileged manner towards some particular rays of the spider web and these change according to its modified electron affinity received instantly from the punctual light source. Since the particle gets directly affected by the time it has spent under the light {\it i.e.} the 'local time' spent by the Brownian particle under the light source of the vertex, this 'local time' has a direct influence on the privilege instantaneous directions elected by the particle ({\it i.e.} the spinning measure of the motion). In turn, such a device would give information on the light scattering of a punctual light source by a single particle. Note that such a device could not be set up by using a classical planar Brownian motion particle : because the trajectories do not have bounded variations it does not seem feasible to pursue the particle with a point laser and such a motion would never return exactly under a fixed punctual light source point. The reader will surely be interested to find in \cite{Clark} what seems to us to be the origins of the study of the relationships between light scattering and Brownian particle motions.

\medskip
\medskip

Let us recall the statement of the main Theorem in \cite{Martinez-Ohavi Walsh} (Theorem 3.1). 

Given
$$(x_\star,i_\star)\in \mathcal{J},~~T>0,$$
then there exists then a unique probability measure denoted by $\P^{x_\star,i_\star}$, defined on the canonical space of continuous maps living on the star-shaped network - times the set of the non negative and non decreasing function - such that for $f$ any regular enough:
$$\big(\mathcal{S}_{pi}-\mathcal{M}_{ar}\big)~~-~~\text{label for the spider martingale problem}:$$
\begin{align} \label{eq:def-V-intro}
&\nonumber\Bigg(f_{i(s)}(s,x(s))- f_{i_\star}(0,x_\star)-\int_{0}^{s}\Big(\partial_tf_{i(u)}(u,x(u))+\displaystyle\frac{1}{2}\sigma_{i(u)}^2(u,x(u),l(u))\partial_{xx}^2f_{i(u)}(u,x(u))\\
&\nonumber\displaystyle+b_{i(u)}(u,x(u),l(u))\partial_xf_{i(u)}(u,x(u))\Big)du - \int_{0}^{s}\big(\displaystyle\sum_{j=1}^{I}\alpha_j(u,l(u))\partial_{x}f_{j}(u,0)\big)dl(u)~\Bigg)_{0\leq s\leq T},
\end{align}
is a martingale under the probability measure $\P^{x_\star,i_\star}$ for the natural filtration generated by the canonical process $(x(s),i(s),l(s))_{s\in [0,T]}$. Here, even if the construction of the process was performed in a more general probability space, classical arguments taken from the Skohokhod's representation of a reflected diffusion ensure that $(l(s))_{s\in [0,T]}$ would still stand for the local time of the expected spider process at the junction point $\bf 0$.

\medskip
\medskip

As we have described previously, the results of this article mainly concern the trajectory properties of the spider built in \cite{Martinez-Ohavi Walsh}. Each of the results stated and proved in this contribution, can be seen {\it mutatis mutandis} as several independent short problems denoted by \textbf{Problem 1 to 6} in the sequel. We introduce them briefly in the next lines of this Introduction, so that the reader may easily read and find his way through this contribution.
\subsection{Problem 1 - It\^o's formula} Using the test function $f:=(x,i)\mapsto x$, we proved in Proposition 5.1 in \cite{Martinez-Ohavi Walsh}, that there exists a $(\Psi_{s})_{0\leq s\leq T}$ standard one dimensional Brownian motion $W$ (depending on $f$), such that for the unique solution $\P^{x_\star,i_\star}$ of $\big(\mathcal{S}_{pi}-\mathcal{M}_{ar}\big)$, we have almost surely $\forall s\in[0,T]$:  
\begin{eqnarray}\label{eq : diff x}
x(s)=x_\star+\displaystyle\int_{0}^{s}b_{i(u)}(u,x(u),l(u))du+\displaystyle\int_{0}^{s}\sigma_{i(u)}(u,x(u),l(u))dW(u)+l(s).
\end{eqnarray}
In the spirit of the It\^o's formula given in \eqref{Ito Sheu}, where the spinning measure is constant, could we also obtain a general It\^o's formula driven by the same Brownian motion $W$? 
\subsection{Problem 2 - Absolute continuity}
The crucial {\it non-stickiness} property of the spider process $(x,i)$ at $\bf 0$, proved in Proposition 5.2 in \cite{Martinez-Ohavi Walsh}, reads:
\begin{eqnarray*}\forall \varepsilon>0,~~\mathbb{E}^{\P^{x_\star,i_\star}}\croc{\int_{0}^{T}\ind{x(u)\leq \varepsilon}ds}\leq C\varepsilon,
\end{eqnarray*}
where $C>0$ is a constant independent of $\varepsilon$. Consequently we have that:
\begin{eqnarray*}
\P^{x_\star,i_\star}\big(x(s)=0\big)=0,~~ds~~\text{a.e in}~~[0,T].  
\end{eqnarray*}
Can we claim that indeed we have:
\begin{eqnarray*}
\forall s\in [0,T],~~\P^{x_\star,i_\star}\big(x(s)=0\big)=0~~?
\end{eqnarray*}
even if the coefficients of diffusion $(\sigma_i,b_i)$ are discontinuous at $\bf 0$, and their discontinuities are driven almost surely by the process $i$, that clearly has a chaotic trajectory with infinite discontinuities. Since the process $x$ has a classical behavior at the interior of each edges, with a density absolutely continuous w.r.t the Lebesgue measure, does this absolute continuity can be extended on the whole star-shaped network ${\mathcal J}$?
\subsection{Problem 3 - Feynman-Kac representation} We have recently proved in our last PDE article \cite{Martinez-Ohavi EDP}, that the following linear parabolic system posed on a star-shaped network:
\begin{eqnarray}\label{eq : pde with l}
\begin{cases}
\textbf{Linear parabolic equation parameterized}\\
\textbf{by the local-time on each ray:}\\
\partial_tu_i(t,x,l)-\sigma_i(t,x,l)\partial_x^2u_i(t,x,l)
+b_i(t,x,l)\partial_xu_i(t,x,l)\\
+c_i(t,x,l)u_i(t,x,l)=f_i(t,x,l),~~(t,x,l)\in (0,T)\times (0,R)\times(0,K),\\
\textbf{Linear local-time Kirchhoff's boundary condition at }\bf 0:\\
\partial_lu(t,0,l)+\displaystyle \sum_{i=1}^I \alpha_i(t,l)\partial_xu_i(t,0,l)=\phi(t,l),~~(t,l)\in(0,T)\times(0,K),\\
\medskip
\textbf{Dirichlet/Neumann boundary conditions outside }  \bf 0:\\
\partial_xu_i(t,R,l)=0,~~ (t,l)\in (0,T)\times(0,K),\\
\forall i\in[\![1,I]\!],~~ u_i(t,x,K)=\psi_i(t,x),~~ (t,x)\in [0,T]\times[0,R],\\
\textbf{Initial condition:}\\
\forall i\in[\![1,I]\!],~~ u_i(0,x,l)=g_i(x,l),~~(x,l)\in[0,R]\times[0,K],\\
\textbf{Continuity condition at } \bf 0:\\
\forall (i,j)\in[\![1,I]\!]^2,~~u_i(t,0,l)=u_j(t,0,l)=u(t,0,l),~~(t,l)\in|0,T]\times[0,K],
\end{cases}
\end{eqnarray}
is well posed in a certain class of regularity, that not ensures completely the continuity of $l \mapsto \partial_lf(t,x,l)$ in the whole domain, but only at $\bf 0$ (see Definition 2.1 in \cite{Martinez-Ohavi EDP}). Does the solutions of \eqref{eq : pde with l}, in the backward formulation, have a Feynman-Kac representation like in the classical cases?
\subsection{Problem 4 - Approximations of the local time at the junction vertex} A local-time of a one dimensional reflected diffusion on the half line has two classical approximations in literature. The first one is called the {\it the "Downcrossing representation of the local time"} ; the second one is an $L^1$ approximation where both the second order term and the average time spent by the process $x$ near $\bf 0$ appear. Do these two representations retain their validity in the case of our spider process $(x,i)$?
\subsection{Problem 5 - Strong Markov property} It is natural to ask if the process solution of the martingale problem $\big(\mathcal{S}_{pi}-\mathcal{M}_{ar}\big)$ satisfies the strong Markov property. 

\subsection{Problem 6 - Diffusion scattering at the junction vertex}  Recall that when the spinning measure is constant, that is:
$$\forall (t,l)\in[0,T]\times[0+\infty),~~\alpha_i(t,l)=\alpha_i,$$
and the coefficients of diffusion are homogeneous: $b_i(t,x,l)=b_i(x),~\sigma_i(t,x,l)=\sigma_i(x)$, it was proved in the seminal work \cite{freidlinS} (see Corollary 2.4) that for $\delta>0$ small enough, if we introduce the following stopping time:
$$\theta^\delta:=\inf\big\{~s\ge 0,~~x(s)=\delta~\big\},$$
then we have:
\begin{eqnarray}\label{eq distr instant diffra 1}
\forall i\in [I],~~\lim_{\delta \searrow 0} \P_{\bf 0}\big(~i(\theta^\delta)=i~\big) =\alpha_i.  
\end{eqnarray}
The last convergence shows that as soon as the spider process $(x,i)$ reaches the junction point $\bf 0$, the 'instantaneous' probability distribution for $(x,i)$ to be scattered along the ray $\mathcal{R}_i$ is exactly equal to $\alpha_i$. Does a similar result remains true for our martingale problem $\big(\mathcal{S}_{pi}-\mathcal{M}_{ar}\big)$? How to formulate it since now the coefficients of diffraction are random?


\medskip
\medskip

\textbf{Organization of the paper}: We will provide an answer to all the \textbf{Problems 1 to 7} previously described, respectively in the \textbf{Sections 3 to 9}.
\section{Notations and a remainder of the main result of \cite{Martinez-Ohavi Walsh}}
Fix $I\geq 2$ an integer. We denote $[I] = \{1,\dots, I\}$ and consider $\mathcal{J}$ a junction space with $I$ edges defined by
\begin{align*}
\mathcal{J}~~:=~~\{{\bf 0}\}\cup \pare{(0,\infty)\times [I]}.
\end{align*}
All the points of $\mathcal{J}$ are described by couples $(x,i)\in [0,\infty)\times [I]$  with the junction point ${\bf 0}$ identified with the equivalent class $\{(0,i)~:~i\in [I]\}$. With a slight abuse of notation, the common junction point ${\bf 0}$ of the $I$ edges will be often denoted be $0$ and we will also often identify the space $\mathcal{J}$ with a union of $I$ edges
$J_i=[0,+\infty)$ satisfying $J_i\cap J_j=\{0\}$ whenever $(i,j)\in [I]^2$ with $i\neq j$. With these notations $(x,i)\in {\mathcal{J}}$ is equivalent to asserting that $x\in J_i$. We endow naturally ${\mathcal{J}}$ with the distance $d^{\mathcal{J}}$ defined by
\begin{equation*}
\forall \Big((x,i),(y,j)\Big)\in \mathcal{J}^2,~~d^\mathcal{J}\Big((x,i),(y,j) \Big)  := \left\{
\begin{array}{ccc}
 |x-y| & \mbox{if }  & i=j\;,\\ 
x+y & \mbox{if }  & {i\neq j},\;
\end{array}\right.
\end{equation*}
so that $\pare{\mathcal{J}, d^{\mathcal{J}}}$ is a Polish space.

Note that we will often write $(x,i,l)$ instead of $((x,i),l)$ an element of ${\mathcal J}\times \R^+$.
\vspace{0,3 cm}
For $T>0$, we introduce the time-space domain $\mathcal{J}_T$ defined by
\begin{align*}\mathcal{J}_T~~&:=~~[0,T]\times\mathcal{J},
\end{align*}
and consider $\mathcal{C}^{\mathcal{J}}[0,T]$ the Polish space of maps defined from $[0,T]$ onto the junction space $\mathcal{J}$ that are continuous w.r.t. the metric $d^{\mathcal{J}}$. The space $\mathcal{C}^{\mathcal{J}}[0,T]$ is naturally endowed with the uniform metric $d^\mathcal{J}_{[0,T]}$ defined by: 
\begin{equation*}
\forall \Big((x,i),(y,j)\Big)\in \pare{\mathcal{C}^{\mathcal{J}}[0,T]}^2,~~d^\mathcal{J}_{[0,T]}:=\sup_{t \in [0,T]}d^\mathcal{J}\Big((x(t),i(t)),(y(t),j(t)) \Big).
\end{equation*}
Together with $\mathcal{C}^{\mathcal{J}}[0,T]$, we introduce
\begin{eqnarray*}
\mathcal{L}[0,T]~~:=~~\Big\{l:[0,T]\to \R^+,\text{ continuous non decreasing}\Big\}
\end{eqnarray*} 
endowed with the usual uniform distance $|\,.\,|_{(0,T)}$.

The modulus of continuity on $\mathcal{C}^{\mathcal{J}}[0,T]$ and $\mathcal{L}[0,T]$ are naturally defined for any $\theta\in (0, T]$ as
\begin{align*}
&\forall X=\big(x,i\big)\in \mathcal{C}^{\mathcal{J}}[0,T],\\
&\hspace{1,3 cm}\omega\pare{X,\theta}= \sup\Big\{d^{\mathcal{J}}\pare{(x(s),i(s)),(x(u),i(u))} ~\big\vert~(u,s)\in [0,T]^2,~~|u-s|\leq \theta\Big\};\\ 
&\forall f\in {\mathcal{L}}[0,T],\\
&\hspace{1,3 cm}\omega\pare{f,\theta}= \sup \Big\{|f(u)-f(s)|~\big\vert~(u,s)\in [0,T]^2,~~|u-s|\leq \theta\Big\}.
\end{align*}

We then form the product space 
\begin{align*}
\Phi~~=~~\mathcal{C}^{\mathcal{J}}[0,T]\times \mathcal{L}[0,T]
\end{align*}
 considered as a measurable Polish space equipped with its Borel $\sigma$-algebra $\mathbb{B}(\Phi)$ generated by the open sets relative to the metric $d^{\Phi}:=d^{\mathcal{J}}_{[0,T]} + |\,.\,|_{(0,T)}$. 

The canonical process $X$ on $(\Omega, \mathcal{F}) := (\Phi, \mathbb{B}(\Phi))$ is defined as
\[X:\begin{array}{cll}
[0,T]\times \Omega&\to\;\;{\mathcal J}\times \R^+\\
(s,\omega) &\mapsto\;\;\tilde{X}(s,\omega) := \omega(s),
\end{array}
\]
where:
$$\omega=(x(s),i(s),l(s))_{s\in [0,T]},~~\text{and}~~\forall s\in[0,T],~~\omega(s)=(x(s),i(s),l(s)).$$
We denote by $(\Psi_t:=\sigma(X(s), 0\leq s \leq t))_{0\leq t\leq T}$ the canonical filtration on $\pare{\Phi, \mathbb{B}(\Phi)}$.
We will denote by $\mathcal{P}([I])\subset [0,1]^I$,   
$
    \mathcal{P}([I]):=\Big\{(\alpha_{i})\in[0,1]^{I}~\big\vert~ \ds \sum_{i=1}^{I}\alpha_i=1\Big\},
$ the simplex set giving all probability measures on $[I]$.

\vspace{0,3 cm}

We introduce the following data  
\vspace{-0,3cm} 
$$\begin{cases}
(\sigma_i)_{i \in \{1 \ldots I\}} \in \pare{\mathcal{C}_b([0,T]\times[0,+\infty)\times [0,+\infty);\R)}^I\\
(b_i)_{i \in \{1 \ldots I\}} \in \pare{\mathcal{C}_b([0,T]\times[0,+\infty)\times [0,+\infty);\R)}^I\\
\alpha=(\alpha_i)_{i \in \{1 \ldots I\}} \in \mathcal{C}([0,T] \times [0,+\infty);\, \mathcal{P}([I]))
\end{cases}
$$
satisfying the following assumption $(\mathcal{H})$ (where $({\bf A})$ stands for alpha, $({\bf E})$ for ellipticity, and $({\bf R})$ for regularity):
\begin{center}{\textbf{Assumption } $(\mathcal{H})$}
\end{center}
\begin{align*}
&({\bf A})~~\exists\;\ds \underline{a}\in \left (0,{1}/{I}\right ],~~\forall i \in \{1    \ldots    I\}, ~~\forall (t,l)\in [0,T]\times[0,+\infty),~~\alpha_i(t,l)~~\ge~~\underline{a}.\\
&({\bf E})~~\exists\,\,\underline{\sigma} >0,~~\forall i \in \{1    \ldots    I\}, ~~\forall (t,x,l)\in [0,T]\times[0,+\infty)\times [0,+\infty),~~\sigma_i(t,x,l)~~\ge~~\underline{\sigma}.\\
&({\bf R})~~\exists (|b|,|\sigma|,|\overline{a}|)\in (0,+\infty)^3,~~\forall i \in \{1 \ldots I\},\\
&({\bf R} - i)\hspace{0,5 cm}\displaystyle\sup_{t,x,l}|b_i(t,x,l)|\;\;\,+\sup_{t,l}\sup_{(x,y),\;x\neq y} \frac{|b_i(t,x,l)-b_i(t,y,l)|}{|x-y|}\\
&\displaystyle\hspace{3,5cm}+\sup_{t,x}\sup_{(l,l'),\;l\neq l'} \frac{|b_i(t,x,l)-b_i(t,x,l')|}{|l-l'|}\\
&\displaystyle
\hspace{3,5 cm}+\sup_{x,l}\sup_{(t,s),\;t\neq s} \frac{|b_i(t,x,l)-b_i(s,x,l)|}{|t-s|}~~\leq~~|b|,\\
&({\bf R} - ii)\hspace{0,5 cm}\displaystyle\sup_{t,x,l}|\sigma_i(t,x,l)|\;\;\,+\sup_{t,l}\sup_{(x,y),\;x\neq y} \frac{|\sigma_i(t,x,l)-\sigma_i(t,y,l)|}{|x-y|}\\
&\displaystyle\hspace{3,5cm}+\sup_{t,x}\sup_{(l,l'),\;l\neq l'} \frac{|\sigma_i(t,x,l)-\sigma_i(t,x,l')|}{|l-l'|}\\
&\displaystyle
\hspace{3,5 cm}+\sup_{x,l}\sup_{(t,s),\;t\neq s} \frac{|\sigma_i(t,x,l)-\sigma_i(s,x,l)|}{|t-s|}~~\leq ~~|\sigma|,\\
&({\bf R}- iii)\hspace{0,5 cm}\displaystyle \sup_t\sup_{(l,l'),\;l\neq l'} \frac{|\alpha_i(t,l)-\alpha_i(t,l')|}{|l-l'|}+\displaystyle\sup_l\sup_{(t,s),\;t\neq s} \frac{|\alpha_i(t,l)-\alpha_i(s,l)|}{|t-s|}~~\leq ~~|\overline{a}|.
\end{align*}
\vspace{0,3 cm}

Let us introduce $\mathcal{C}^{1,2,1}_b(\mathcal{J}_T\times[0+\infty)])$ the class of continuous function defined on $\mathcal{J}_T\times[0+\infty)$ with regularity $\mathcal{C}^{1,2,1}_b([0,T]\times [0,\infty)^2)$ on each edge, namely
\begin{align*}
&\mathcal{C}^{1,2}_b(\mathcal{J}_T\times[0+\infty)):=\Big\{f:\mathcal{J}_T\to\R,~~(t,(x,i),l)\mapsto f_i(t,x,l)~\Big\vert~~ \forall i\in [I],\\
&f_i:[0,T]\times {J}_i\times[0+\infty)\to\R,\,(t,x,l)\mapsto f_i(t,x,l)\in \mathcal{C}^{1,2,1}_b([0,T]\times J_i\times[0+\infty)),\\
&\hspace{3 cm}~ \forall (t,(i,j),l)\in [0,T]\times [I]^{2}\times[0+\infty), \,f_i(t,0,l)=f_j(t,0,l)\Big\}.
\end{align*}
In the same way, we define $\mathcal{C}^{1,2}_b(\mathcal{J}_T)$, removing the dependence w.r.t the variable $l$.\\
\\
\textbf{Main result obtained in \cite{Martinez-Ohavi Walsh} (see Theorem 3.1)}:\\
\\
We end this Section by recalling the main result of the companion of this paper, related to the existence and uniqueness of weak solutions for a class of spider diffusions with random selections depending on the own local time of the process at the junction point.
\\
Define the following martingale problem of $(\Phi,\mathbb{B}(\Phi))$:
\smallskip
\begin{center}
\emph{$\big(\mathcal{S}_{pi}-\mathcal{M}_{ar}\big)$}
\end{center}
{\it Fix a given terminal condition $T>0$ and $(x_\star,i_\star)\in \mathcal{J}$. Can we ensure existence and uniqueness of a probability $\P^{x_\star,i_\star}$ defined on the measurable space $(\Phi,\mathbb{B}(\Phi))$ such that:\\
-(i) $\big(x(0),i(0),l(0)\big)=\big(x_\star,i_\star,0\big)$, $\P^{x_\star,i_\star}$-a.s.\\ 
-(ii) For each $s\in [0,T]$: 
\begin{eqnarray*}
\displaystyle\displaystyle\int_{0}^{s}\ind{x(u)>0}dl(u)=0,~~\P^{x_\star,i_\star}-\text{ a.s.}
\end{eqnarray*}
and $(l(u))_{u\in [0,T]}$ has increasing paths $\P^{x_\star,i_\star}$-almost surely.\\
-(iii) For any $f\in\mathcal{C}^{1,2}_b(\mathcal{J}_T)$, the following process:
\begin{align}
\label{eq:def-V1}
&\Bigg(~f_{i(s)}(s,x(s))- f_{i}(0,x)\displaystyle\\
&\displaystyle -\int_{0}^{s}\pare{\partial_tf_{i(u)}(u,x(u))+\displaystyle\frac{1}{2}\sigma_{i(u)}^2(u,x(u),l(u))\partial_{xx}^2f_{i(u)}(u,x(u))}du\nonumber\\
&\displaystyle-\int_{0}^{s}\pare{b_{i(u)}(u,x(u),l(u))\partial_xf_{i(u)}(u,x(u))}du-\displaystyle\sum_{j=1}^{I}\int_{0}^{s}\alpha_j(u,l(u))\partial_{x}f_{j}(u,0)dl(u)~\Bigg)_{0\leq s\leq T}\nonumber,
\end{align} 
is a $(\Psi_s)_{0 \leq s \leq T}$ martingale under the measure of probability $\P^{x_\star,i_\star}$.}
\smallskip
Hence, the main result given in \cite{Martinez-Ohavi Walsh}, Theorem 3.1 reads:
\begin{Theorem}\label{th: exis Spider}
Assume assumption $(\mathcal{H})$. Then, the martingale problem $\big(\mathcal{S}_{pi}-\mathcal{M}_{ar}\big)$ is well-posed.
\end{Theorem}
Finally, we state the following corollary of the last Theorem \ref{th: exis Spider}, that will be useful in this contribution.
\begin{Corollary}\label{cr: exis Spider}
Assume assumption $(\mathcal{H})$. Fix a given terminal condition $T>0$, $t\in[0,T]$ and $(x_\star,i_\star,l_\star)\in \mathcal{J}\times[0,+\infty)$. Then there exists a unique probability $\P_t^{x_\star,i_\star,l_\star}$ defined on the measurable space $(\Phi,\mathbb{B}(\Phi))$ such that:\\
-(i) $\big(x(s),i(s),l(s)\big)=\big(x_\star,i_\star,l_\star\big)$, for all $s\in[0,t]$, $\P_t^{x_\star,i_\star,l_\star}$-a.s.\\ 
-(ii) For each $s\in [t,T]$: 
\begin{eqnarray*}
\displaystyle\displaystyle\int_{t}^{s}\ind{x(u)>0}dl(u)=0,~~\P_t^{x_\star,i_\star,l_\star}-\text{ a.s.}
\end{eqnarray*}
and $(l(u))_{u\in [t,T]}$ has increasing paths $\P_t^{x_\star,i_\star,l_\star}$-almost surely.\\
-(iii) For any $f\in\mathcal{C}^{1,2,1}_b(\mathcal{J}_T\times[0+\infty)])$, the following process:
\begin{align}
\label{eq:def-V}
&\nonumber\Bigg(~f_{i(s)}(s,x(s).l(s))- f_{i_\star}(t,x_\star,l_\star)\\
&\nonumber\displaystyle -\int_{t}^{s}\pare{\partial_tf_{i(u)}(u,x(u),l(u))+\displaystyle\frac{1}{2}\sigma_{i(u)}^2(u,x(u),l(u))\partial_{xx}^2f_{i(u)}(u,x(u),l(u))}du\\
&\nonumber\displaystyle-\int_{t}^{s}\pare{b_{i(u)}(u,x(u),l(u))\partial_xf_{i(u)}(u,x(u),l(u))}du\\
&-\ds\int_{t}^{s}\Big(\partial_{l}f(u,0,l(u))+\displaystyle\sum_{j=1}^{I}\alpha_j(u,l(u))\partial_{x}f_{j}(u,0,l(u))\Big)dl(u)~\Bigg)_{t\leq s\leq T},
\end{align} 
is a $(\Psi_s)_{t \leq s \leq T}$ martingale under the measure of probability $\P_t^{x_\star,i_\star,l_\star}$.   
\end{Corollary}



\section{Problem 1 - It\^o's formula}
We answer to the first problem - \textbf{Problem 1} - described in Introduction \ref{sec intro compa}. We obtain then an It\^o's formula, with exactly the same Brownian  motion $W_f(\cdot)=W(\cdot)$, appearing in the stochastic dynamic \eqref{eq : diff x}, satisfied by the process $x(\cdot)$.
\begin{Theorem}\label{th It\^o's formula}
  For any $f\in\mathcal{C}^{1,2,1}_b(\mathcal{J}_T\times[0+\infty)])$,
\begin{align}
\label{eq:Ito-formula 00}
&f_{i(t)}(t,x(t),l(t))- f_{i_\star}(0,x_\star,0)\displaystyle\nonumber\\
&=\displaystyle \int_{0}^{t}\sigma_{i(u)}(u,x(u),l(u))\partial_xf_{i(u)}(u,x(u),l(u))dW(u)\nonumber\\
&\;\;\;\displaystyle + \int_{0}^{t}\partial_tf_{i(u)}(u,x(u),l(u))du+\int_{0}^{t}b_{i(u)}(u,x(u),l(u))\partial_xf_{i(u)}(u,x(u),l(u))du\nonumber\\
&\;\;\;\displaystyle +\frac{1}{2}\int_{0}^{t}\sigma_{i(u)}^2(u,x(u),l(u))\partial_{xx}^2f_{i(u)}(u,x(u),l(u))du\nonumber\\
&\;\;\;\displaystyle + \int_{0}^{t}\partial_{l}f(u,0,l(u))dl(u) + \displaystyle\sum_{j=1}^I\int_{0}^{t}\alpha_j(u,l(u))\partial_{x}f_{j}(u,0,l(u))dl(u)
\end{align} 
holding $\P^{x_\star,i_\star}-\text{ a.s.}$ for any $t\in[0,T]$.
\end{Theorem}

\begin{proof}

In the following and in order to avoid overloading the notations unnecessarily, all equalities and inequalities have to be understood in the  $\P^{x_\star,i_\star} - \text{ a.s.}$ sense. 
\medskip
\medskip

For any $h\in\mathcal{C}^{2}_b(\mathcal{J})$, $u\in [0,T]$ and $l\in[0,+\infty)$, let us introduce the notation
\begin{align*}
\mathcal{L}^{u,l}_{i}h_{i}(x) = b_{i}(u,x,l)\partial_xh_{i}(x) +\frac{1}{2}\sigma_{i}^2(u,x,l)\partial_{xx}^2h_{i}(x)
\end{align*}
for $i\in [\![1,I]\!]$.

To begin with let us fix $f\in\mathcal{C}^{2}_b(\mathcal{J})$.

{\bf Step 1}

Set 
\begin{align}
{\mathcal M}_t(f)&:=f_{i(t)}(x(t)) - f(x_\star)\nonumber\\
&\;\;\;\;- \int_{0}^{t} \mathcal{L}^{s,l(s)}_{i(s)}f_{i(s)}(x(s))ds -\displaystyle\sum_{j=1}^I\int_{0}^{t}\alpha_j(s,l(s))\partial_{x}f_{j}(0)dl(s).
\end{align}

From the fundamental theorem, we know that $\pare{{\mathcal M}_t(f)}_{t\in [0,T]}$ is an $\pare{{\mathcal F}_t}$-martingale.

Observe that $t\mapsto \ds f(x_\star) + \int_{0}^{t} \mathcal{L}^{s,l(s)}_{i(s)}f_{i(s)}(x(s))ds + \displaystyle\sum_{j=1}^I\int_{0}^{t}\alpha_j(s,l(s))\partial_{x}f_{j}(0)dl(s)$ is of bounded variation. Thus we have
\begin{equation}
\label{eq:quadratic-variation-m2}
\langle {\mathcal M}(f) \rangle_t = \langle f_{i(.)}(x(.))\rangle_t.
\end{equation}

The aim of this first step is to study the quadratic variation of $\pare{{\mathcal M}(f)}_{t\in [0,T]}$. For this purpose we fix some $\varepsilon >0$ and introduce the following sequence of stopping times related to the excursions of the process $x(\cdot)$ around the junction point $\bf 0$:
\begin{eqnarray*}
&\tau_0^\varepsilon=0\\   
&\theta_0^\varepsilon:=\inf\{s\ge 0 :~~x(s)=\varepsilon\}\\
&\tau_1^\varepsilon=\inf\{s\ge \theta_0^\varepsilon :~~x(s)=0\}\\
&..................\\
&\theta_n^\varepsilon:= \inf\{s\ge \tau_{n}^\varepsilon :~~x(s)=\varepsilon\}\\
&\tau_{n+1}^\varepsilon:= \inf\{s\ge \theta_{n}^\varepsilon : ~~x(s)=0\}\\
&..................
\end{eqnarray*}

We start from the decomposition
\begin{align}
f_{i(t)}(x(t)) - f_{i_\star}(x_\star) &= \sum_{n\geq 0} f_{i(t\wedge\theta_n^\varepsilon)}(x(t\wedge\theta_n^\varepsilon)) - f_{i(t\wedge\tau_n^\varepsilon)}(x(t\wedge\tau_n^\varepsilon))\nonumber\\
&\hspace{0,4 cm}+ \sum_{n\geq 0} f_{i(t\wedge\tau_{n+1}^\varepsilon)}(x(t\wedge\tau_{n+1}^\varepsilon)) - f_{i(t\wedge\theta_n^\varepsilon)}(x(t\wedge\theta_n^\varepsilon))\nonumber\\
&:=\sum_{n\geq 0} M_t^{1,n,\varepsilon}(f) + \sum_{n\geq 0} M_t^{2,n,\varepsilon}(f).
\end{align}
\medskip

Observe that $\P^{x_\star,i_\star}$ is concentrated by definition on ${\mathcal C}(\mathcal {J}\times [0,+\infty))$ and the topology given on $\mathcal {J}$ induces that $s\mapsto i(s)$ remains constant on each interval $[\theta_n^\varepsilon, \tau_{n+1}^\varepsilon]$.
In particular we have 
\begin{align}
\label{eq:cases-decomp-ta}
&M_t^{2,n,\varepsilon}(f):= f_{i(t\wedge\tau_{n+1}^\varepsilon)}(x(t\wedge\tau_{n+1}^\varepsilon)) - f_{i(t\wedge\theta_n^\varepsilon)}(x(t\wedge\theta_n^\varepsilon))\nonumber\\
&=\left \{\begin{array}{l}
f_{i(\tau_{n+1}^\varepsilon)}((x(\tau_{n+1}^\varepsilon)) - f_{i(\theta_n^\varepsilon)}(x(\theta_n^\varepsilon)) = f_{i(\theta_n^\varepsilon)}((x(\tau_{n+1}^\varepsilon)) - f_{i(\theta_n^\varepsilon)}(x(\theta_n^\varepsilon))\;\;\;\text{ if }t\geq \tau^\varepsilon_{n+1}\\
f_{i(t)}((x(t)) - f_{i(\theta_n^\varepsilon)}(x(\theta_n^\varepsilon)) = f_{i(\theta_n^\varepsilon)}((x(t)) - f_{i(\theta_n^\varepsilon)}(x(\theta_n^\varepsilon))\;\;\;\text{ if }t\in [\theta_n^\varepsilon, \tau^\varepsilon_{n+1})\\
0\;\;\;\text{ if }t<\theta_n^\varepsilon.
\end{array}
\right .
\end{align}
There is also no increase of the local time on each interval $[\theta_n^\varepsilon, \tau_{n+1}^\varepsilon]$. So, conditionally on $i(\theta_n^\varepsilon)= j$, we may apply the classical It\^o's formula applied to $f_{j}$ and $(x(t))$ (whose differential is given in \eqref{eq : diff x}) and from \eqref{eq:cases-decomp-ta} we prove that 
\begin{align*}
M_t^{2,n,\varepsilon}(f)&= \int_{t\wedge \theta_n^\varepsilon}^{t\wedge\tau_{n+1}^\varepsilon} \mathcal{L}_{i(s)}f_{i(\theta_n^\varepsilon)}(x(s))ds + \int_{t\wedge\theta_n^\varepsilon}^{t\wedge\tau_{n+1}^\varepsilon}\partial_{x}f_{i(\theta_n^\varepsilon)}(x(s))\sigma_{i(s)}(s,x(s),l(s))dW(s)\\
&=\int_{t\wedge \theta_n^\varepsilon}^{t\wedge\tau_{n+1}^\varepsilon} \mathcal{L}_{i(s)}f_{i(s)}(x(s))ds + \int_{t\wedge\theta_n^\varepsilon}^{t\wedge\tau_{n+1}^\varepsilon}\partial_{x}f_{i(s)}(x(s))\sigma_{i(s)}(s,x(s),l(s))dW(s).
\end{align*}
In particular
\begin{align}
\label{eq:vq-facile0}
\langle M^{2,n,\varepsilon}(f)\rangle_t
&=\int_{t\wedge\theta_n^\varepsilon}^{t\wedge\tau_{n+1}^\varepsilon}\pare{\sigma_{i(s)}(s,x(s),l(s))\partial_{x}f_{i(s)}(x(s))}^2 ds.
\end{align}
Due to the independence of the increments of the Brownian motion on the intervals $[\theta_n^\varepsilon,\tau_{n+1}^\varepsilon)$ and $ [\theta_{n'}^\varepsilon,\tau_{n'+1}^\varepsilon)$, we have using Kronoecker's symbol
\begin{align*}
&\langle \sum_{n\geq 0}M^{2,n,\varepsilon}(f)\rangle_t = \langle \sum_{n\geq 0}M^{2,n,\varepsilon}(f), \sum_{n'\geq 0}M^{2,n',\varepsilon}(f)\rangle_t\\
&=\sum_{n\geq 0}\sum_{n'\geq 0} \langle \int_{.\wedge\theta_n^\varepsilon}^{.\wedge\tau_{n+1}^\varepsilon}\partial_{x}f_{i(s)}(x(s))\sigma_{i(s)}(s,x(s),l(s))dW(s),\\
&\ds  \int_{.\wedge\theta_{n'}^\varepsilon}^{.\wedge\tau_{n'+1}^\varepsilon}\partial_{x}f_{i(s)}(x(s))\sigma_{i(s)}(s,x(s),l(s))dW(s)\rangle_t\\
&=\sum_{n\geq 0}\sum_{n'\geq 0}\langle M^{2,n,\varepsilon}(f), M^{2,n',\varepsilon}(f)\rangle_t \delta_{n\,n'}=\sum_{n\geq 0}\langle M^{2,n,\varepsilon}(f)\rangle_t,
\end{align*}
from which we deduce that
\begin{align}
\label{eq:vq-facile}
&\nonumber \langle \sum_{n\geq 0}M^{2,n,\varepsilon}(f)\rangle_t = \sum_{n\geq 0}\langle M^{2,n,\varepsilon}(f)\rangle_t =\\
&\ds\sum_{n\geq 0}\int_{t\wedge\theta_n^\varepsilon}^{t\wedge\tau_{n+1}^\varepsilon}\pare{\sigma_{i(s)}(s,x(s),l(s))\partial_{x}f_{i(s)}(x(s))}^2 ds.
\end{align}
Let us now turn to the study of the quadratic variation of the sum $\ds \pare{\sum_n M_t^{1,n,\varepsilon}(f)}_{t\in [0,T]}$ {\it i.e.} the quadratic variation of the sum $$\pare{\sum_n f_{i(s\wedge \theta_n^\varepsilon)}(x(s\wedge \theta_n^\varepsilon)) - f_{i(s\wedge\tau_n^\varepsilon)}(x(s\wedge\tau_n^\varepsilon))}_{s\in [0,T]}$$
that we decide to write more informally
$$\pare{\sum_n [\Delta_n(f)]_s}_{s\in [0,T]}\text{ to simplify the notations.}$$ 

Introduce an arbitrary discretization $(s_{i})_{i\in T}$ of $[0,t]$. We have
\begin{align*}
{J}\pare{{\mathcal T}}
&:=\sum_{i\in T}\pare{\sum_{n}([\Delta_{n}(f)]_{s_{i+1}} -  [\Delta_{n}(f)]_{s_{i}})}^2\\
&= \sum_{i\in T}\sum_n\pare{[\Delta_n(f)]_{s_{i+1}} - [\Delta_n(f)]_{s_{i}}}^2\\
&\;\;\;\;\;- \sum_{i\in T}\sum_{n,n',n\neq n'}\pare{[\Delta_n(f)]_{s_{i+1}} - [\Delta_n(f)]_{s_{i}}}\pare{[\Delta_{n'}(f)]_{s_{i+1}} - [\Delta_{n'}(f)]_{s_{i}}}.
\end{align*}
Observe that due to $[\tau_n, \theta_n]\cap[\tau_{n'}, \theta_{n'}]=\emptyset$ when $n\neq n'$, we have
$$\sum_{i\in T}\sum_{n,n',n\neq n'}\pare{[\Delta_n(f)]_{s_{i+1}} - [\Delta_n(f)]_{s_{i}}}\pare{[\Delta_{n'}(f)]_{s_{i+1}} - [\Delta_{n'}(f)]_{s_{i}}}=0.$$
Hence,
\begin{align*}
{J}\pare{{\mathcal T}}&= \sum_{i\in T}\pare{\sum_n [\Delta_n(f)]_{s_{i+1}} - [\Delta_n(f)]_{s_{i}}}^2\\
&= \sum_{i\in T}\sum_n\pare{[\Delta_n(f)]_{s_{i+1}} - [\Delta_n(f)]_{s_{i}}}^2\\
&=\sum_n \sum_{i\in T}\pare{f_{i(s_{i+1}\wedge \theta^\varepsilon_n)}(x(s_{i+1}\wedge \theta^\varepsilon_n)) - f_{i(s_{i}\wedge \theta^\varepsilon_n)}(x(s_{i}\wedge \theta^\varepsilon_n)}^2\ind{s_{i+1}\in [\tau^\varepsilon_n,\theta^\varepsilon_n]},
\end{align*}
so that
\begin{align*}
&J\pare{{\mathcal T}}\leq 2\sum_n \sum_{i\in T}\pare{f_{i(s_{i+1}\wedge \theta_n)}(x(s_{i+1}\wedge \theta^\varepsilon_n)) - f(0)}^2\ind{i(s_{i+1}\wedge \theta^\varepsilon_n)\neq i(s_{i}\wedge \theta^\varepsilon_n)}\ind{s_{i+1}\in [\tau^\varepsilon_n,\theta^\varepsilon_n]}\\
&+ 2\sum_n \sum_{i\in T}\pare{f_{i(s_{i}\wedge \theta^\varepsilon_n)}(x(s_{i}\wedge \theta^\varepsilon_n)- f(0)}^2\ind{i(s_{i+1}\wedge \theta^\varepsilon_n)\neq i(s_{i}\wedge \theta^\varepsilon_n)}\ind{s_{i+1}\in [\tau^\varepsilon_n,\theta^\varepsilon_n]}\\
&+ \sum_n \sum_{i\in T}\pare{f_{i(s_{i+1}\wedge \theta^\varepsilon_n)}(x(s_{i+1}\wedge \theta^\varepsilon_n)) - f_{i(s_{i}\wedge \theta^\varepsilon_n)}(x(s_{i}\wedge \theta^\varepsilon_n)}^2\ind{i(s_{i+1}\wedge \theta^\varepsilon_n)= i(s_{i}\wedge \theta^\varepsilon_n)}\ind{s_{i+1}\in [\tau^\varepsilon_n,\theta^\varepsilon_n]}\\
&\leq C(f)^2\sum_n \sum_{i\in T}\pare{x(s_{i+1}\wedge \theta^\varepsilon_n) - x(s_{i}\wedge \theta^\varepsilon_n)}^2\ind{s_{i+1}\in [\tau^\varepsilon_n,\theta^\varepsilon_n]}.
\end{align*}
But 
\begin{align*}
\lim_{|T|\rightarrow 0}\sum_{i\in T}\pare{x(s_{i+1}\wedge \theta^\varepsilon_n)- x(s_{i}\wedge \theta^\varepsilon_n)}^2\ind{s_{i+1}\in [\tau^\varepsilon_n,\theta^\varepsilon_n]} = \int_{\tau^\varepsilon_n}^{\theta^\varepsilon_n}\sigma_{i(s)}^2(u,x(u),l(u))du
\end{align*} and
\begin{align*}
&\P^{x_\star,i_\star} -\lim_{|T|\rightarrow 0}\croc{\sum_n\sum_{i\in T}\pare{x(s_{i+1}\wedge \theta^\varepsilon_n) - x(s_{i}\wedge \theta^\varepsilon_n)}^2\ind{s_{i+1}\in [\tau^\varepsilon_n,\theta^\varepsilon_n]}}\\
&\hspace{1,5 cm}= \int_{0}^{t}\sigma_{i(s)}^2(u,x(u),l(u))\ind{x(u)\leq \varepsilon}du.
\end{align*}
So that 
\begin{align}
\label{eq:vq-difficile}
\langle \sum_n M^{1,n,\varepsilon}(f)\rangle_t &= \P^{x_\star,i_\star}-\limsup_{|T|\rightarrow 0}\croc{\sum_{i\in T}\pare{\sum_n [\Delta_n(f)]_{s_{i+1}} - [\Delta_n(f)]_{s_{i}}}^2}\nonumber\\
&\leq  C(f)^2\int_{0}^{t}\ind{x(u)\leq \varepsilon}du.
\end{align}
Thus, from the fact $[\tau_n, \theta_n)\cap[\theta_{n}, \tau_{n+1}]=\emptyset$, we have
\begin{align*}
&\langle {\mathcal M}(f)\rangle_t - \int_0^t \pare{\sigma_{i(s)}(s,x(s),l(s))\partial_{x}f_{i(s)}(x(s))}^2 ds\\
&\hspace{1,5 cm}= \langle\sum_{n\geq 0} M^{1,n,\varepsilon}(f) + \sum_{n\geq 0} M^{2,n,\varepsilon}(f)\rangle_t - \int_0^t \pare{\sigma_{i(s)}(s,x(s),l(s))\partial_{x}f_{i(s)}(x(s))}^2 ds\\
&\hspace{1,5 cm}= \langle\sum_{n\geq 0} M^{1,n,\varepsilon}(f)\rangle_t + \langle \sum_{n\geq 0} M^{2,n,\varepsilon}(f)\rangle_t - \int_0^t \pare{\sigma_{i(s)}(s,x(s),l(s))\partial_{x}f_{i(s)}(x(s))}^2 ds
\end{align*}
and from \eqref{eq:vq-facile} and \eqref{eq:vq-difficile} we deduce
\begin{align*}
&\left |\langle {\mathcal M}(f)\rangle_t - \int_0^t \pare{\sigma_{i(s)}(s,x(s),l(s))\partial_{x}f_{i(s)}(x(s))}^2 ds\right |\\
&\hspace{1,5 cm}\leq C(f)^2\int_{0}^{t}\ind{x(u)\leq \varepsilon}du+ \left |\sum_{n\geq 0}\int_{t\wedge\theta_n^\varepsilon}^{t\wedge\tau_{n+1}^\varepsilon}\pare{\sigma_{i(s)}(s,x(s),l(s))\partial_{x}f_{i(s)}(x(s))}^2 ds\right .\\
&\hspace{2.0 cm}\left .- \int_0^t \pare{\sigma_{i(s)}(s,x(s),l(s))\partial_{x}f_{i(s)}(x(s))}^2 ds\right |\\
&\hspace{2.0 cm}= C(f)^2\int_{0}^{t}\ind{x(u)\leq \varepsilon}du + \left | \sum_{n\geq 0}\int_{t\wedge\tau_n^\varepsilon}^{t\wedge\theta_{n}^\varepsilon}\pare{\sigma_{i(s)}(s,x(s),l(s))\partial_{x}f_{i(s)}(x(s))}^2 ds\right |\\
&\hspace{1,5 cm}\leq 2C(f)^2\int_{0}^{t}\ind{x(u)\leq \varepsilon}du.
\end{align*}
Sending $\varepsilon$ to zero and using the non-stickiness condition, we prove that the quadratic variation of $\pare{{\mathcal M}(f)}_{t\in [0,T]}$ is given for all $t\in [0,T]$ by
\begin{equation}
\label{eq:quadratic-variation-m}
\langle {\mathcal M}(f)\rangle_t = \ds \int_0^t \pare{\sigma_{i(s)}(s,x(s),l(s))\partial_{x}f_{i(s)}(x(s))}^2 ds,
\end{equation}
which concludes our first step.

\medskip

{\bf Step 2}

Set $t\mapsto m_t:=\ds x(t) - x_\star - \int_0^t b_{i(u)}(u,x(u),l(u))du - l(t) = \ds\int_{0}^{t}\sigma_{i(s)}(s,x(s),l(s))dW(s)$.

Observe that $\{m_t, t\in [0,T]\}$ is an $({\mathcal F}_t)$-martingale.
\[
\langle m\rangle_t = \int_0^t \sigma^2_{i(s)}(s,x(s),l(s))dW(s) = \langle x\rangle_t.
\]
Again with similar computations as in {\bf Step 1} and using the non-stickiness condition, we prove the following equality
\begin{equation}
\label{eq:quadratic-variation-m0}
\langle m, {\mathcal M}(f)\rangle_t = \langle x, {\mathcal M}(f)\rangle_t = \int_0^t \sigma_{i(s)}(s,x(s),l(s))\partial_{x}f_{i(s)}(x(s)) ds
\end{equation}
holding for any $t\in [0,T]$ and which gives the mutual-quadratic variation between $(m_t)_{t\in [0,T]}$ and $\pare{{\mathcal M}_t}_{t\in [0,T]}$.

Since the computations are very similar to {\bf Step 1} we only sketch the argument without getting into details.

Using the same notations as {\bf Step 1}, we have 
\begin{align*}
\langle m, {\mathcal M}(f)\rangle_t &= \langle \sum_{n\geq 0} M^{1,n,\varepsilon}(\id) + \sum_{n\geq 0} M^{2,n,\varepsilon}(\id),\sum_{n\geq 0} M^{1,n,\varepsilon}(f) + \sum_{n\geq 0} M^{2,n,\varepsilon}(f)\rangle_t\\
&= \langle \sum_{n\geq 0} M^{1,n,\varepsilon}(\id), \sum_{n\geq 0} M^{1,n,\varepsilon}(f)\rangle_t + \langle\sum_{n\geq 0} M^{2,n,\varepsilon}(\id), \sum_{n\geq 0} M^{2,n,\varepsilon}(f)\rangle_t\\
&=\sum_{n\geq 0} \langle M^{1,n,\varepsilon}(\id), M^{1,n,\varepsilon}(f)\rangle_t + \sum_{n\geq 0}\langle M^{2,n,\varepsilon}(\id), M^{2,n,\varepsilon}(f)\rangle_t.
\end{align*}

From the classical It\^o' formula applied to $f_j$ and $(x(t))$ we find
\begin{equation}
\langle M^{2,n,\varepsilon}(\id), M^{2,n,\varepsilon}(f)\rangle_t = \int_{t\wedge\theta_n^\varepsilon}^{t\wedge\tau_{n+1}^\varepsilon}\sigma^2_{i(s)}(s,x(s),l(s))\partial_{x}f_{i(s)}(x(s))ds,
\end{equation}
whereas we have the estimate
\begin{align}
\label{eq:vq-difficile}
 &\sum_n \langle M^{1,n,\varepsilon}(\id), M^{1,n,\varepsilon}(f) \rangle_t\nonumber\\
 &= \P^{x_\star,i_\star}-\limsup_{|T|\rightarrow 0}\croc{\sum_{i\in T}\pare{\sum_n \pare{[\Delta_n(f)]_{s_{i+1}} - [\Delta_n(f)]_{s_{i}}}}\pare{\pare{[\Delta_n(\id)]_{s_{i+1}} - [\Delta_n(\id)]_{s_{i}}}}}\nonumber\\
&\leq  C(f)\int_{0}^{t}\ind{x(u)\leq \varepsilon}du.\nonumber
\end{align}
Hence, we deduce
\begin{align*}
&\left |\langle m, {\mathcal M}(f)\rangle_t - \int_0^t \sigma^2_{i(s)}(s,x(s),l(s))\partial_{x}f_{i(s)}(x(s)) ds\right |\leq 2C(f)\int_{0}^{t}\ind{x(u)\leq \varepsilon}du.
\end{align*}
Sending $\varepsilon$ to zero and using the non-stickiness condition, we deduce \eqref{eq:quadratic-variation-m0}.
\medskip

{\bf Step 3}

Set 
\[[Q_t(f)(m, {\mathcal M})]:=
\pare{\begin{array}{cc}
\langle m\rangle_t&\langle m, {\mathcal M}(f)\rangle_t\\
\langle m, {\mathcal M}(f)\rangle_t &\langle{\mathcal M}(f)\rangle_t
\end{array}}
\]
together with $\Theta_t(f):=\pare{
\partial_x f(x(t)), -1}$.

Since both $(m_t)_{t\in [0,T]}$ and $\pare{{\mathcal M}_t(f)}_{t\in [0,T]}$ are $\pare{{\mathcal F}_t}$-martingales with bounded quadratic variation, we have that for any $\lambda\in \R$
\begin{eqnarray*}
&\ds t\mapsto {\mathcal E}_t(f)\!\!:=\exp\Big(\lambda \int_0^t <\Theta_s(f), d(m_s, {\mathcal M}_s(f))> -\\
& \ds  \frac{\lambda^2}{2}\int_0^t <\Theta_s(f),[dQ_s(f)(m, {\mathcal M})]\Theta_s(f)>\Big)
\end{eqnarray*}
defines an exponential $\pare{{\mathcal F}_t}$-martingale.

Now observe from the definition and from the results of Steps 1 and 2 that \[<\Theta(f),\;[Q(f)(m, {\mathcal M})]\Theta(f)>\equiv 0.\] From the martingale property written for $\pare{{\mathcal E}_t(f)}_{t\in [0,T]}$, this in turn implies that 
\[\int_0^.<\Theta_s(f), d(m_s, {\mathcal M}_s(f))>\equiv 0.\]
Remembering the definition of $(m_t)_{t\in [0,T]}$ and making use of \eqref{eq : diff x}, this last equality is nothing but a rewriting of It\^o's formula for $f\in\mathcal{C}^{2}_b(\mathcal{J})$.
\medskip

{\bf Step 4} Extension to $f\in \mathcal{C}^{1,2,1}_b(\mathcal{J}_T\times[0+\infty)])$.

Let $f\in\mathcal{C}^{1,2,1}_b([0,T]\times \mathcal{J}\times \R^+;\R)$ with product form $f(t,x,l) = t^k\,h(x)\,l^m$. Applying the classical It\^o formula to $f$, since the local time only increases at times $s$ where $x(s)=0$ we have
\begin{align}
\label{eq:Ito-formula}
&f_{i(t)}(t,x(t),l(t))- f_{i_\star}(0,x_\star,0)\displaystyle\nonumber\\
&=\int_0^t ks^{k-1}h(x(s))l(s)^m ds + \int_0^t s^{k}l(s)^m\,d[h(x(s))] +\int_0^t m s^{k}h(x(0))l^{m-1}(s)dl(s) 
\end{align} 
and using the differential of $h(x(s))$ deduced from {\bf Step 3}, we get \eqref{eq:Ito-formula 00} for such $f$.

From this ascertainment, the remaining arguments are routine. We refer for e.g. to \cite{Revuz-Yor} Chapter IV Proof of Theorem (3.3) for an exposition of these.
\end{proof}

\section{Problem 2 - Absolute continuity}
In this section we apply the ideas of \cite{Fournier-Printems} and prove the absolute continuity of the law of $x(t)$ w.r.t. Lebesgue's measure. 

\subsection{Preliminaries}
Let us apply the previous It\^o's formula to the following multifunction
\[
(f_1, \dots, f_I)~:~(t,x,l)\mapsto \pare{-\int_0^x\frac{dy}{\sigma_1(t,y,l)},\,\dots,\,-\int_0^x\frac{dy}{\sigma_I(t,y,l)}}.
\]
We obtain the differentiation
\begin{align}
\label{eq:Ito-formula1}
&f_{i(t)}(t,x(t),l(t))\displaystyle\nonumber\\
&=\displaystyle f_{i_\star}(0,x_\star,0)\,-\,W(t)\nonumber\\
&\;\;\;\displaystyle + \int_{0}^{t}\int_0^{x(u)}\frac{\partial_u{\sigma_{i(u)}}(u,y,l(u))}{\sigma^2_{i(u)}(u,y,l(u))}dy\,du - \int_{0}^{t}\frac{b_{i(u)}(u,x(u),l(u))}{\sigma_{i(u)}(u,x(u),l(u))}du\nonumber\\
&\;\;\;\displaystyle +\frac{1}{2}\int_{0}^{t}\partial_{x}\sigma_{i(u)}(u,x(u),l(u))du\nonumber\;+\;\displaystyle\sum_{j=1}^I\int_{0}^{t}\frac{\alpha_j(u,l(u))}{\sigma_{j}(u,0+,l(u))}dl(u)\\
&:=y_0 - W(t) + \int_0^t h_{i(u)}(u,x(u),l(u))du + \ell(t)
\end{align} 
holding $\P^{x_\star,i_\star}-\text{ a.s.}$ for any $t\in[0,T]$. Above we introduced the notations
\begin{align*}
h_{i}(t,x,l) &:= \int_0^{x}\frac{\partial_t{\sigma_{i}}(t,y,l)}{\sigma^2_{i}(t,y,l)}dy -\frac{b_{i}(t,x,l)}{\sigma_{i}(t,x,l)}\displaystyle +\frac{1}{2}\partial_{x}\sigma_{i}(t,x,l)\,\,\,(i\in \{1,\dots, I\})\\
&\text{and}\\
\ell(t)&:=\sum_{j=1}^I\int_{0}^{t}\frac{\alpha_j(u,l(u))}{\sigma_{j}(u,0+,l(u))}dl(u),\;\;\;\;y_0:=f_{i_\star}(0,x_\star,0).
\end{align*}

Observe that under our assumptions there exists $C>0$ s.t.
\begin{align*}
|h^2_{i}(t,x,l)| \leq C(1+x^2).
\end{align*}

\subsection{No atom at zero}
Let us consider the random variable 
\[Y(t):=\displaystyle y_0 - W(t) + \int_0^t h_{i(u)}(u,x(u),l(u))du.\]

We would like to the results of \cite{Fournier-Printems} on the regularity of the law of $Y(t)$ w.r.t. Lebesgue's measure. Note that we are not exactly in the case of application of Theorem 3.1 of Fournier-Printems \cite{Fournier-Printems}: indeed -- due to the presence of the chaotic $u\mapsto i(u)$ in the drift coefficient $u\mapsto h_{i(u)}(u,x(u),l(u))$ -- the random variable $Y(t)$ does not quite meet the conditions of the theorem. However, $u\mapsto h_{i(u)}(u,x(u),l(u))$ has all the good integrability properties needed and we prove that, as long as $t>0$, the law of the random variable $Y(t)$ possesses a density. 

Let us be a bit more specific, by a scaling argument, we can reduce to the case $t=1$. 

Let $Z_\varepsilon := y_0 - W(1-\varepsilon)$.
\begin{align*}
\displaystyle\E\pare{(Y(1) - Z_\varepsilon)^2}&\leq 2\varepsilon^2 + \displaystyle2\E\pare{\int_{1-\varepsilon}^1 h(i(u),u,x(u),l(u))du}^2\\
&\leq 2\varepsilon^2 + \displaystyle2\varepsilon\E\pare{\int_{1-\varepsilon}^1 h^2(i(u),u,x(u),l(u))du}\\
&\leq 2\varepsilon^2 + \displaystyle C\varepsilon\int_{1-\varepsilon}^1 \E(1 + \sup_{u\in [0,1]}|x(u)|^2)du\\
&\leq C\varepsilon^2.
\end{align*}

We can then proceed as in Fournier-Pintems \cite{Fournier-Printems} to prove that ${\mathcal L}\pare{Y(1)}$ possesses a density w.r.t Lebesgue's measure, in particular it does not have an atom at zero.
Since $\ell(1)\geq 0$, $\P^{x_\star,i_\star}-\text{ a.s.}$, this property extends to the law of $Y(1) + \ell(1) = f_{i(1)}(t,x(1),l(1))$ which also does not have an atom at zero. But taking a closer look at the definition of $(f_i)_{i\in \{1,\dots, I\}}$ this means that $\mathcal{L}(x(1))$ itself cannot have an atom at zero. This conclusion generalizes to $\mathcal{L}(x(t))$ for any $t>0$.
\subsection{Existence of a density}
We observe that $s\mapsto \ell(s)$ defined above increases only on the zero set $\{s\in [0,T]:x(s)=0\}$ of $(x(s))_{s\in [0,T]}$, which is also the zero set of $(f_{i(s)}(t,x(s),l(s)) = Y(s) + \ell(s))_{s\in [0,T]}$: in particular 
$$
\int_0^t (Y(s) + \ell(s))d\ell(s) = 0,\;\;\;\;\P^{y_0}-a.s.
$$
Hence, we have
\begin{align*}
f^{2}_{i(t)}(t,x(t),l(t))&=(Y(t) + \ell(t))^2\\
&=f^{2}_{i(0)}(t,x(0),l(0)) + 2\int_0^t (Y(s) + \ell(s))dY(s) \\
&+ 2\int_0^t (Y(s) + \ell(s))d\ell(s) + t\\
&=f^{2}_{i(0)}(t,x(0),l(0)) + 2\int_0^t (Y(s) + \ell(s))dW(s)\\
&+ 2\int_0^t (Y(s) + \ell(s))h_{i(s)}(s,x(s),l(s))ds + t.
\end{align*}
Consequently, we see that $\{V(s):=f^{2}_{i(s)}(s,x(s),l(s)) : s\in [0,T]\}$ satisfies
an SDE with random coefficients that writes
\begin{align*}
V(t) = V(0) + \int_0^t 2\sqrt{V(s)}dW(s) + \int_0^t 2\sqrt{V(s)}h_{i(s)}(s,x(s),l(s))ds + t.
\end{align*}

By using the method of Fournier-Printems \cite{Fournier-Printems} we then prove that $V(t)$ possesses a density on $\R^+\setminus \{0\}$ and with a possible atom at $\{0\}$. But this ladder case is excluded because of the discussion of the previous paragraph.

\section{Problem 3 - Feynman-Kac representation}
The main purpose of this section, is to give a probabilistic representation (Feynman-Kac's formula) of some solutions for backward parabolic system posed on star-shaped networks - having the local time Kirchhoff's boundary transmission - introduced in our PDE contribution \cite{Martinez-Ohavi EDP}.

In the whole section, \textbf{we work under assumption ($\mathcal{H}$)}. We fix a terminal condition $T>0$, and:
$$(t,x,i,l)\in[0,T)\times \mathcal{J}\times[0+\infty).$$
We introduce the unique probability measure $\P_t^{x,i,l}$ defined on the measurable canonical space $(\Phi,\mathbb{B}(\Phi))$, solution of the spider martingale problem $\big(\mathcal{S}_{pi}-\mathcal{M}_{ar}\big)$, stated in Corollary \ref{cr: exis Spider}.
Furthermore, we introduce the following data:
$$\begin{cases}
\Big(h_i \in \mathcal{C}\big([0,T]\times [0,+\infty)^2,\R\big)\Big)_{i\in[\![1,I]\!]}\\
h_0 \in \mathcal{C}\big([0,T]\times [0,+\infty),\R\big)\\
g\in \mathcal{C}\big( \mathcal{J} \times [0,+\infty),\R\big)
\end{cases},
$$
satisfying the following assumption:
\begin{align*}
&\exists |h|\in (0,\infty),~~\forall i\in[\![1,I]\!]:~~\sup \Big\{|h_i(t,x,l)|+\frac{|h_i(t,x,l)-h_i(t,x,l)|}{|t-s|+|x-y|+|l-q|},\\
&~~(t,s,x,y,l,q)\in [0,T]^2\times [0,+\infty)^4,~~t\neq s,x\neq y,l\neq q\Big\}\leq |h|,\\
&\sup \Big\{|h_0(t,l)|+\frac{|h_0(t,l)-h_0(t,l)|}{|t-s|+|l-q|},~~(t,s,l,q)\in [0,T]^2\times [0,+\infty)^2,~~t\neq s,l\neq q\Big\}\leq |h|.
\end{align*}
For all $i\in[\![1,I]\!]$, the terminal condition $(x,l)\mapsto g_i(x,l)$ belongs to $ \mathcal{C}^{2,0}_b\big((0,+\infty)^2,\R\big)$, whereas the map $l\mapsto g(0,l)\in\mathcal{C}^1_b\big([0,+\infty),\R\big)$, and the following compatibility condition holds true:
\begin{align*}
&\partial_lg(0,l)+\sum_{i=1}^I\alpha_i(T,l)\partial_xg_i(0,l)+h_0(T,l)=0,~~l\in[0,+\infty).
\end{align*}
Let us recall the class of regularity required for the solutions of the PDE systems having {\it local time Kirchhoff's boundary condition at the vertex}, introduced in \cite{Martinez-Ohavi EDP} for the first time (see Definition 2.1) for a bounded star-shaped network, and extended to unbounded domain in \cite{Martinez-Ohavi Walsh} (see Definition 6.1). 
\begin{Definition} (of the class $\mathfrak{C}^{1,2,0}_{\{0\},b} \big(\mathcal{J}_T\times[0,+\infty)\big)$ \label{def: class regula EDP}.) We say that
$$f:=\begin{cases} \mathcal{J}_T\times[0,+\infty)\to \R,\\
\big(t,(x,i),l\big)\mapsto f_i(t,x,l)
\end{cases}.$$
is in the class $f\in \mathfrak{C}^{1,2,0}_{\{0\},b} \big(\mathcal{J}_T\times[0,+\infty)\big)$ if\\
(i) the following continuity condition holds at the junction point $\bf 0$: \\for all $(t,l)\in[0,T]\times[0,+\infty)$, for all $(i,j)\in [I]^2$, $f_i(t,0,l)=f_j(t,0,l)=f(t,0,l)$;\\
(ii) for all $i\in [I]$, the map $(t,x,l)\mapsto f_i(t,x,l)$ has a regularity in the class\\ $\mathcal{C}^{0,1,0}_b\big([0,T]\times [0,+\infty)^2,\R \big)$;\\
(iii) for all $i\in [I]$, the map $(t,x,l)\mapsto f_i(t,x,l)$ has regularity in the interior of each ray $\mathcal{R}_i$ in the class $\mathcal{C}^{1,2,0}_b\big((0,T)\times (0,+\infty)^2,\R \big)$;\\
(iv) at the junction point $\bf 0$, the map $(t,l)\mapsto f(t,0,l)$ has a regularity in the class\\ $\mathcal{C}^{0,1}_b\big([0,T]\times [0,+\infty),\R \big)$;\\
(v) for all $i\in [I]$, on each ray $\mathcal{R}_i$, $f$ admits a generalized locally integrable derivative with respect to the variable $l$ in $\displaystyle \bigcap\limits_{q\in (1,+\infty)}\!L^q_{loc}\big((0,T)\times (0,+\infty)^2\big)$.
\end{Definition}
Now we state the main result of this section, that reads:
\begin{Theorem}\label{th : exis para with l bord infini}
The unique solution of the following backward linear parabolic system involving a {\it local time Kirchhoff's transmission condition} posed on the domain $\mathcal{J}_T\times[0,+\infty)$:
\begin{eqnarray}\label{eq : EDP Feynmann Kac}
\begin{cases}&\partial_tu_i(t,x,l)+\ds \frac{1}{2}\sigma^2_i(t,x,l)\partial_x^2u_i(t,x,l)\\
&\hspace{2,0 cm}+b_i(t,x,l)\partial_xu_i(t,x,l)+h_i(t,x,l)=0,\,(t,x,l)\in (0,T)\times (0,+\infty)^2,\\
&\partial_lu(t,0,l)+\displaystyle \sum_{i=1}^I \alpha_i(t,l)\partial_xu_i(t,0,l)+h_0(t,l)=0,~~(t,l)\in(0,T)\times(0,+\infty),\\
&\forall (i,j)\in[I]^2,~~u_i(t,0,l)=u_j(t,0,l)=u(t,0,l),~~(t,l)\in[0,T]\times[0,+\infty)^2,\\
&\forall i\in[I],~~ u_i(T,x,l)=g_i(x,l),~~(x,l)\in[0,+\infty)^2,
\end{cases}
\end{eqnarray}
is given for all $ (t,x,l)\in[0,T]\times[0,+\infty)^2,~\forall  i\in[I]$, by:
\begin{align*}
u_i(t,x,l)=\mathbb{E}^{\P_t^{x,i,l}}\Big[\int_t^Th_{i(s)}\big(s,x(s),l(s)\big)ds+\int_t^Th_{0}\big(s,l(s)\big)dl(s)+g_{i(T)}\big(x(T),l(T)\big)\Big].  \end{align*}
\end{Theorem}
\begin{proof}
Recall that the class of test function $\mathcal{C}_b^{1,2,1}(\mathcal{J}_T\times[0+\infty))$ used for the statement of our martingale problem $\big(\mathcal{S}_{pi}-\mathcal{M}_{ar}\big)$, is stronger than the class $\mathfrak{C}^{1,2,0}_{\{0\},b} \big(\mathcal{J}_T\times[0,+\infty)\big)$. However, we have managed to prove that the martingale property given in Theorem \ref{th: exis Spider} and therefore also in Corollary \ref{cr: exis Spider}, extends further to any test function in the class $\mathfrak{C}^{1,2,0}_{\{0\},b} \big(\mathcal{J}_T\times[0,+\infty)\big)$ (see Proposition 6.4 in \cite{Martinez-Ohavi Walsh}.) Let then: $$\Biggl(\Phi,\mathbb{B}(\Phi),(\Psi_s)_{0 \leq s \leq T}, \Big(x(\cdot),i(\cdot),l(\cdot)\Big),\P_t^{x,i,l}\Biggl)$$
be the weak solution solution of the spider martingale problem $\big(\mathcal{S}_{pi}-\mathcal{M}_{ar}\big)$, given in Theorem \ref{th: exis Spider}, starting at time $t$ at the point $(x,i,l)$. Applying the martingale property to the unique solution $u$ of system \eqref{eq : EDP Feynmann Kac}, between time $t$ and $T$, we obtain:
\begin{align*}
&\mathbb{E}^{\P_t^{x,i,l}}\Big[u_{i(T)}\big(T,x(T),l(T)\big)\Big]=u_i(t,x,l)\\
&+\mathbb{E}^{\P_t^{x,i,l}}\Big[~~\displaystyle\int_{t}^{T}\Big(\partial_tu_{i(s)}(s,x(s),l(s))+\displaystyle\frac{1}{2}\sigma_{i(u)}^2(s,x(s),l(s))\partial_{x}^2u_{i(s)}(s,x(s),l(s))\\
&+b_{i(s)}(s,x(s),l(s))\partial_xu_{i(s)}(s,x(s),l(s))\Big)ds\\
&+\ds\int_{t}^{T}\Big(\partial_lu(s,0,l(s))+\displaystyle\sum_{j=0}^{I}\alpha_j(s,l(s))\partial_{x}u_{i}(s,0,l(s))\Big)dl(s)~~\Big]\\
&\ds =u_i(t,x,l)-\mathbb{E}^{\P_t^{x,i,l}}\Big[\int_t^Th_{i(s)}\big(s,x(s),l(s)\big)ds+\int_t^Th_{0}\big(s,l(s)\big)dl(s)\Big]\\
&+\mathbb{E}^{\P_t^{x,i,l}}\Big[~~\displaystyle\int_{t}^{T}\Big(\partial_tu_{i(s)}(s,x(s),l(s))+\displaystyle\frac{1}{2}\sigma_{i(u)}^2(s,x(s),l(s))\partial_{x}^2u_{i(s)}(s,x(s),l(s))\\
&+b_{i(s)}(s,x(s),l(s))\partial_xu_{i(s)}(s,x(s),l(s))+h_{i(s)}\big(s,x(s),l(s)\big)\Big)ds\\
&+\ds\int_{t}^{T}\Big(\partial_lu(s,0,l(s))+\displaystyle\sum_{j=0}^{I}\alpha_j(s,l(s))\partial_{x}u_{i}(s,0,l(s))+h_{0}\big(s,l(s)\big)\Big)dl(s)~~\Big].
\end{align*} 
Using now the PDE system \eqref{eq : EDP Feynmann Kac} satisfied by the unique solution $u$; together with the terminal condition, it follows that for all $ (t,x,l)\in[0,T]\times[0,+\infty)^2,~\forall  i\in[I]$:
 \begin{align*}
u_i(t,x,l)=\mathbb{E}^{\P_t^{x,i,l}}\Big[\int_t^Th_{i(s)}\big(s,x(s),l(s)\big)ds+\int_t^Th_{0}\big(s,l(s)\big)dl(s)+g_{i(T)}\big(x(T),l(T)\big)\Big].  \end{align*} 
\end{proof}
\begin{Remark}\label{rm Feynman Kac elliptique}
We could in an analogous way formulate corresponding Feynman-Kac's type representations in the elliptical framework. 
\end{Remark}

\section{Problem 4 - Quadratic approximations of the local time}\label{sec approx quadra temps loc}

In this section we obtain two types of approximation for the local time process $l(\cdot)$. The main key will be once again the use of the {\it non stickiness} condition that reads for the unique solution $\P^{x_\star,i_\star}$ of our spider martingale problem $\big(\mathcal{S}_{pi}-\mathcal{M}_{ar}\big)$ (given in Theorem \ref{th: exis Spider})~:
\begin{equation}\label{eq : maj zero x}
\forall \varepsilon>0,~~\mathbb{E}^{\P^{x_\star,i_\star}}\Big[~~\int_{0}^{T}\ind{x(u)\leq \varepsilon}ds~~\Big]~~\leq~~C\varepsilon,
\end{equation}
where $C>0$ is an uniform constant, depending on the data of the system introduced in assumption $(\mathcal{H})$, (see Proposition 5.2 in \cite{Martinez-Ohavi Walsh}). 

The next proposition extends - one of - the results obtained originally by Paul Levy for a reflected simple one dimensional Brownian motion, to our Walsh's spider process constructed in \cite{Martinez-Ohavi Walsh}. This approximation is generally called in literature : {\it the "Downcrossing representation of the local time"}, see Theorem 2.23, Chapter VI in \cite{Karatzas Book} (P. Levy's theory of the Brownian local time), for a reference and on this subject.
\begin{Proposition}\label{pr : approx int temps local}
Assume $x(0)=x_\star>0$, and fix $\varepsilon>0$. Define the following sequence of stopping times, that characterize the excursions of the process $x(\cdot)$ around the junction point $\bf 0$, having a (small) length $\varepsilon$:
\begin{eqnarray*}
&\tau_0^\varepsilon=0;\\
&\text{ and recursively for } n\ge 1 ;\\
&\theta_1^\varepsilon:=\inf\{t>\tau_0^\varepsilon,~~x(s)=\varepsilon\};\\
&\tau_1^\varepsilon=\inf\{ t> \theta_1^\varepsilon,~~x(s)=0\};\\
&..................\\
&\theta_n^\varepsilon:= \inf\{t> \tau_{n-1}^\varepsilon,~~x(s)=\varepsilon\};\\
&\tau_{n}^\varepsilon:= \inf\{t>\theta_{n}^\varepsilon,~~x(s)=0\};\\
&..................
\end{eqnarray*}
If we define  $\forall \varepsilon > 0,~\forall t\in [0,T]$, the following (right continuous with left limit) process $\mathcal{N}^\varepsilon(\cdot)$ by: 
\begin{align*}
\mathcal{N}^\varepsilon(t)&:=\text{all the number of times that } x(s): s < t \text{ crosses down from } \varepsilon \text{ to } 0,\\ 
&:=\sup\Big\{~~n\ge 1,~~\text{such that:}~~[\theta_n^\varepsilon, \tau_{n}^\varepsilon]\subset [0,t]~~\Big\},~~ \P^{x_\star,i_\star}~~\text{a.s},
\end{align*}
we will get: 
\begin{align}\label{eq appro quadra 1}
\forall t\in [0,T],~~\lim_{\varepsilon \searrow 0} \E^{\P^{x_\star,i_\star}}\Big[~~\Big|~\varepsilon \mathcal{N}^\varepsilon(t)-l(t)~\Big|~~\Big]=0.
\end{align}
Moreover, for any map $f$ belonging to the class $\mathfrak{C}^{1,2,0}_{\{0\},b} \big(\mathcal{J}_T\times[0,+\infty)\big)$ (defined in Definition \ref{def: class regula EDP}), we have $\forall t\in [0,T]$:
\begin{align}\label{eq appro quadra 2}
&\nonumber \lim_{\varepsilon \searrow 0} \E^{\P^{x_\star,i_\star}}\Big[~~\Big|~~\ds \sum_{n=1}^{\mathcal{N}^\varepsilon(t)} \Big(f_{i(\theta_{n+1}^\varepsilon )}(\theta_{n+1}^\varepsilon,x(\theta_{n+1}^\varepsilon),\ell(\theta_{n+1}^\varepsilon))-f_{i(\tau_{n}^\varepsilon)}(\tau_{n}^\varepsilon,x(\tau_{n}^\varepsilon),\ell(\tau_{n}^\varepsilon))\Big)-\\&\ds\int_0^{t}\Big(\partial_lf(s,0,l(s))+\ds \sum_{i=1}^I\alpha_i(s,l(s))\partial_xf_i(s,0,l(s))\Big)dl(s)~~\Big|\mathbf{1}_{\{\mathcal{N}^\varepsilon(t)\ge 1\}}~~\Big]=0.  
\end{align}
\end{Proposition}
\begin{proof}
We focus in proving only \eqref{eq appro quadra 2}, since it appears clear that \eqref{eq appro quadra 1} can be obtained with the same arguments, with the aid of the idendity map and $f=Id$, after an argument of localization.

\textbf{Recall that $x(0)=x_\star>0$, and we can assume without lose of generality that $\varepsilon<<x_\star$.} 
For simplicity, we denote in the rest of the proof:
$$F(\cdot)=f_{i(\cdot)}(\cdot,x(\cdot),l(\cdot)),~~\P^{x_\star,i_\star}~~\text{a.s,}$$
whereas on each edge, $\mathcal{L}[f]$ is given by:
$$\forall (s,x,i,l);~~\mathcal{L}[f](s,x,i,l)=\partial_tf_i(s,x,l)+\frac{1}{2}\sigma_i^2(s,x,l)\partial^2_xf_i(s,x,l)+b_i(s,x,l)\partial_xf_i(s,x,l),$$
Let $t\in [0,T]$. Using the It\^o's formula given in Theorem \ref{th It\^o's formula} in the present contribution and that the paths of the local time process $l(\cdot)$ are flat on $\displaystyle \bigcup_{p\ge 1}[\theta_{p}^\varepsilon,\tau_{p}^\varepsilon)$
we have:
\begin{align*}
&F(t)-F(0)=\int_{0}^{ t}\partial_xf_{i(s)}(s,x(s),l(s))\sigma_{i(s)}(s,x(s),l(s))dW(s)\\
&\ds\;\;\;\;\;\;+\int_0^{t}\mathcal{L}[f](s,x(s),i(s),l(s))ds+
 \int_0^{t}\Big(\partial_lf(s,0,l(s))+\ds\sum_{i=1}^I\alpha_i(s,l(s))\partial_xf_i(s,0,l(s))\Big)dl(s)\\
&=\sum_{n\ge 1}F(\tau_{n}^\varepsilon \wedge t)-F(\theta_n^\varepsilon \wedge t)+\sum_{n\ge 1}F(\theta_{n}^\varepsilon \wedge t)-F(\tau_{n-1}^\varepsilon \wedge t)\\
&=\sum_{n\ge 1}\int_{\theta_n^\varepsilon \wedge t}^{\tau_{n}^\varepsilon \wedge t}\partial_xf_{i(s)}(s,x(s),l(s))\sigma_{i(s)}(s,x(s),l(s))dW(s)\\
&\ds+ \sum_{n\ge 1}\int_{\theta_n^\varepsilon \wedge t}^{\tau_{n}^\varepsilon \wedge t}\mathcal{L}[f](s,x(s),i(s),l(s))ds +\sum_{n\ge 1}F(\theta_{n}^\varepsilon \wedge t)-F(\tau_{n-1}^\varepsilon \wedge t),~~\P^{x_\star,i_\star}~~\text{a.s.}
\end{align*}
We obtain therefore:
\begin{align}\label{eq : dec 1 appro quadr}
\nonumber &\sum_{n\ge 1}F(\theta_{n+1}^\varepsilon \wedge t)-F(\tau_{n}^\varepsilon \wedge t)\\
&\nonumber \ds - \Big(\ds \int_0^{t}\big(\partial_lf(s,0,l(s))+\ds\sum_{i=1}^I\alpha_i(s,l(s))\partial_xf_i(s,0,l(s))\big)dl(s)\Big)\\
&\nonumber =\sum_{n\ge 1}\int_{\tau_n^\varepsilon \wedge t}^{\theta_{n+1}^\varepsilon \wedge t}\partial_xf_{i(s)}(s,x(s),l(s))\sigma_{i(s)}(s,x(s),l(s))\ind{x(u)\leq \varepsilon}dW(s)\\
&+\sum_{n\ge 1}\int_{\tau_n^\varepsilon \wedge t}^{\theta_{n+1}^\varepsilon \wedge t}\mathcal{L}[f](s,x(s),i(s),l(s))\ind{x(u)\leq \varepsilon}ds,~~\P^{x_\star,i_\star}~~\text{a.s.}
\end{align}
In the last equation \eqref{eq : dec 1 appro quadr}, with the aid of the assumptions on the coefficients $(\mathcal{H})$ and the test function $f$, using that the intervals $[\tau_n^\varepsilon \wedge t,\theta_{n+1}^\varepsilon \wedge t]$ are distinct, we obtain that there exists a constant $C>0$ independent of $\varepsilon$, such that:
\begin{align*}
&\E^{\P^{x_\star,i_\star}}\Big[~~\Big(~~\sum_{n\ge 1}\int_{\tau_n^\varepsilon \wedge t}^{\theta_{n+1}^\varepsilon \wedge t } \partial_xf_{i(s)}(s,x(s),l(s))\sigma_{i(s)}(s,x(s),l(s))\ind{x(u)\leq \varepsilon}dW(s)~~\Big)^2~~\Big]=\\
&\E^{\P^{x_\star,i_\star}}\Big[~~\int_0^t\sum_{n\ge 1}\mathbf{1}_{[\tau_n^\varepsilon \wedge t,\theta_{n+1}^\varepsilon \wedge t]}(s)\big(\partial_xf_{i(s)}(s,x(s),l(s))\sigma_{i(s)}(s,x(s),l(s))\big)^2ds~~\Big]\leq\\
& C\E^{\P^{x_\star,i_\star}}\Big[~~\int_0^T\ind{0\leq x(u) \leq \varepsilon}du~~\Big].   \end{align*}
Therefore, from the  {\it non-stickiness} property \eqref{eq : maj zero x}:
\begin{align*}
&\lim_{\varepsilon \searrow 0}\E^{\P^{x_\star,i_\star}}\Big[~~\Big(~~\sum_{n\ge 1}\int_{\tau_n^\varepsilon \wedge t}^{\theta_{n+1}^\varepsilon \wedge t } \partial_xf_{i(s)}(s,x(s),l(s))\sigma_{i(s)}(s,x(s),l(s))\ind{x(u)\leq \varepsilon}dW(s)~~\Big)^2~~\Big]=0,   \end{align*}
and from Cauchy-Schwarz inequality:
\begin{align*}
&\lim_{\varepsilon \searrow 0}\E^{\P^{x_\star,i_\star}}\Big[~~\Big|~~\sum_{n\ge 1}\int_{\tau_n^\varepsilon \wedge t}^{\theta_{n+1}^\varepsilon \wedge t } \partial_xf_{i(s)}(s,x(s),l(s))\sigma_{i(s)}(s,x(s),l(s))\ind{x(u)\leq \varepsilon}dW(s)~~\Big|~~\Big]=0.   \end{align*}
With the same arguments, we have:
\begin{align*}
&\lim_{\varepsilon \searrow 0}\E^{\P^{x_\star,i_\star}}\Big[~~\Big|~~\sum_{n\ge 1}\int_{\tau_n^\varepsilon \wedge t}^{\theta_{n+1}^\varepsilon \wedge t}\mathcal{L}[f](s,x(s),i(s),l(s))\ind{x(u)\leq \varepsilon}ds~~\Big|~~\Big]=0.   \end{align*}
We obtain then:
\begin{align}\label{eq cv 1 tezmps local}
\nonumber&\lim_{\varepsilon \searrow 0}\E^{\P^{x_\star,i_\star}}\Big[~~\Big|~~\sum_{n\ge 1}F(\theta_{n+1}^\varepsilon \wedge t)-F(\tau_{n}^\varepsilon \wedge t)-\\
&\Big(\ds \int_0^{t}\big(\partial_lf(s,0,l(s))+\ds\sum_{i=1}^I\alpha_i(s,l(s))\partial_xf_i(s,0,l(s))\big)dl(s)\Big)~~\Big|~~\Big]=0.   
\end{align}
Assume now first that there exists $p\in \mathbb{N}$ such that $t\in [\theta_{p+1}^\varepsilon,\tau_{p+1}^\varepsilon)$. In this case, we have $\mathcal{N}^\varepsilon(t)=p$ and:
\begin{align*}
&\sum_{n\ge 1}F(\theta_{n+1}^\varepsilon \wedge t)-F(\tau_{n}^\varepsilon \wedge t)=\sum_{n= 1}^{\mathcal{N}^\varepsilon(t)}F(\theta_{n+1})-F(\tau_{n}^\varepsilon).
\end{align*} 
On the other hand, if there exists $p\in \mathbb{N}$ such that $t \in [\tau_{p}^\varepsilon,\theta_{p+1}^\varepsilon)$, we have:
\begin{align*}
\sum_{n\ge 1}F(\theta_{n+1}^\varepsilon \wedge t)-F(\tau_{n}^\varepsilon \wedge t)=F(t)-F(\tau_{p}^\varepsilon)+\Big[\sum_{n=1}^{\mathcal{N}^\varepsilon(t)}F(\theta_{n+1}^\varepsilon)-F(\tau_{n}^\varepsilon)\Big]\ind{\mathcal{N}^\varepsilon(t)\ge 1}.
\end{align*}
We deduce therefore that to obtain:
\begin{align*}
\nonumber&\lim_{\varepsilon \searrow 0}\E^{\P^{x,i}}\Big[~~\Big|~~\sum_{n=1}^{\mathcal{N}^\varepsilon(t)}F(\theta_{n+1}^\varepsilon)-F(\tau_{n}^\varepsilon)-\\
&\Big(\ds \int_0^{t}\big(\partial_lf(s,0,l(s))+\ds\sum_{i=1}^I\alpha_i(s,l(s))\partial_xf_i(s,0,l(s))\big)dl(s)\Big)~~\Big|\ind{\mathcal{N}^\varepsilon(t)\ge 1}~~\Big]=0,  
\end{align*}
it is enough to prove:
\begin{align}\label{eq cv erreur excursion}
\lim_{\varepsilon \searrow 0}\E^{\P^{x,i}}\Big[~~\Big|~~\sum_{p\ge 1}\big(F(t)-F(\tau_{p}^\varepsilon)\big)\mathbf{1}_{\{t\in [\tau_{p}^\varepsilon,\theta_{p+1}^\varepsilon)\}}~~\Big|~~\Big]=0. 
\end{align}
Remark that from the ellipticity assumption we have $\forall p\ge 1$:
\begin{align*}
&\varepsilon^2=\E^{\P_{t}^{x,i,l}}\big[\big(x(\theta_{p+1}^\varepsilon)-x(\tau_{p}^\varepsilon)\big)^2\big]=\E^{\P_{t}^{x,i,l}}\big[\int_{\tau_{p}^\varepsilon}^{\theta_{p+1}^\varepsilon}\sigma_{i(s)}(s,x(s),l(s))^2ds\big]\\
&\ge c~\E^{\P_{t}^{x,i,l}}\big[\theta_{p+1}^\varepsilon-\tau_{p}^\varepsilon \big].
\end{align*}
Hence $\forall p\ge 1$:
\begin{align*}
    \lim_{\varepsilon\searrow 0}\E^{\P_{t}^{x,i,l}}\big[\theta_{p+1}^\varepsilon-\tau_{p}^\varepsilon \big]=0.
\end{align*}
We can then conclude that \eqref{eq cv erreur excursion} holds true using the regularity of $f$, the continuity of the paths of the canonical process, and Lebesgue's theorem.
\end{proof}
\begin{Corollary}
Assume that $x(0)=x_\star$. Then \eqref{eq appro quadra 1}  and \eqref{eq appro quadra 2} hold true.
\end{Corollary}
\begin{proof}
We can use the same arguments of the last Proposition for any $t\in [\theta_1^\varepsilon,T]$. Because:
$$\lim_{\varepsilon \searrow 0}\theta_1^\varepsilon=0,~~\theta_1^\varepsilon \ge0,~~\P^{x,i}~~\text{a.s},$$
we can conclude using that the process $\mathcal{N}^\varepsilon(\cdot)$ is right continuous.
\end{proof}
\begin{Remark}\label{rm Freidlin loi grand nombre}
The  result obtained in the last proposition, more precisely the approximation \eqref{eq appro quadra 2}, generalizes one of the key point used in \cite{freidlinS}. Indeed, the construction of the spider with constant spinning measure introduced in \cite{Freidlin-Wentzell-2}, with the aid of semi-group theory, implies the strong markovian property. Therefore, it follows that the that sequence of random variables $\big(i(\theta_n^\varepsilon)_{n\ge 0}\big)$ are i.i.d. This last property was the key to obtain the It\^o's formula in Lemma 2.3 of \cite{freidlinS}. More precisely, with the aid of the law of large number, one can show (for $f$ regular enough) that:
$$\lim_{\varepsilon \searrow 0}  \sum_{n}\Big(f_{i(\theta_{n+1}^\varepsilon )}(x(\theta_{n+1}^\varepsilon))-f_{i(\tau_{n}^\varepsilon)}(x(\tau_{n}^\varepsilon))\Big)=\Big(\sum_{i=1}^I\alpha_i\partial_xf_i(0)\Big)l(t).$$
Hence, the approximation \eqref{eq appro quadra 2} extends this convergence in our case where the random variables $\big(i(\theta_n^\varepsilon)_{n\ge 0}\big)$ are not i.i.d.
\end{Remark}
The second result of this Section is the following mean-value approximation for the local time $l(\cdot)$ at the junction point $\bf 0$.
\begin{Proposition}
For any nonempty subset $\mathcal{K}\subset [I]$, we have: 
\begin{eqnarray}\label{eq : appro quadra sous ens}
&\nonumber\ds \lim_{\varepsilon \searrow 0} \E^{\P^{x_\star,i_\star}}\Big[~\Big|\sum_{j\in \mathcal{K}}\int_0^{\cdot}\alpha_j(s,l(s))dl(s)-\\
&\ds \sum_{j\in \mathcal{K}}\frac{1}{2\varepsilon}\int_0^{\cdot}\sigma_j^2(s,0,l(s))\ind{0\leq x(s) \leq \varepsilon, i(s)=j}ds\Big|_{(0,T)}~\Big]=0. 
\end{eqnarray}
In particular,
\begin{eqnarray}\label{eq : appro quadra sous}
\lim_{\varepsilon \searrow 0} \E^{\P^{x_\star,i_\star}}\Big[~\Big|l(\cdot)-\sum_{j\in [I]}\frac{1}{2\varepsilon}\int_0^{\cdot}\sigma_j^2(s,0,l(s))\ind{0\leq x(s) \leq \varepsilon, i(s)=j}ds\Big|_{(0,T)}~\Big]=0.\end{eqnarray}
\end{Proposition}
\begin{proof}
The proof will be achieved in two steps. Given $\varepsilon>0$, we introduce first the following function:
\begin{eqnarray}
    \phi^\varepsilon:=\begin{cases}\mathcal{J}\to \R,\\
    (x,i)\mapsto \begin{cases}\frac{x^2}{2\varepsilon},~~\text{if}~~x\leq \varepsilon,\\
    x-\frac{\varepsilon}{2},~~\text{if}~~x\ge \varepsilon.\end{cases}\end{cases}
\end{eqnarray}
We will focus in getting \eqref{eq : appro quadra sous}, since \eqref{eq : appro quadra sous ens} can be obtained with the same arguments considering the same map $\phi^\varepsilon$, but vanishing
on each edge whose indexes belong to $[I]\setminus \mathcal{K}$.

\textbf{Step 1}: We claim first that $\P^{x_\star,i_\star}~~\text{a.s}$:
\begin{equation*}
\begin{split}
&\Big (\phi^\varepsilon(x(s))- \phi^\varepsilon(x_\star)=\\
&\ds
\displaystyle\int_{0}^{s}\displaystyle\frac{1}{2}\sigma_{i(u)}^2(u,x(u),l(u))\partial_x^2\phi^\varepsilon(x(u))
+b_{i(u)}(u,x(u),l(u))\partial_x\phi^\varepsilon(x(u))du+\\
&\int_0^s\sigma_{i(u)}(u,x(u),l(u))\partial_x\phi^\varepsilon(x(u))dW(u)+\displaystyle\sum_{i=1}^{I}\int_{0}^{s}\alpha_i(u,l(u))\partial_x\phi^\varepsilon(0)dl(u)\Big)_{0\leq s \leq T}.
\end{split}
\end{equation*}
The main points to obtain the last formula is to treat the discontinuity of the second derivative of $\phi^\varepsilon$ at $x=\varepsilon$ and to argue by localization. For this purpose, remark first that we can show using the same ideas of Proposition 5.2 in \cite{Martinez-Ohavi Walsh} (modifying in the proof the indicator function appearing (15) for the EDO system (16), by an indicator function of a small ball around the neighborhood of $x=\varepsilon$), that:
\begin{eqnarray}\label{eq : non stick a epislon}
\E^{\P^{x_\star,i_\star}}\Big[~~\int_0^{T}\mathbf{1}_{\{x(s)=\varepsilon\}}ds~~\Big]=0.  
\end{eqnarray}
In the sequel, we are going to regularize $\phi^\varepsilon$ by convolution. Let $\theta >0$. We introduce $\rho_\theta$ an infinite differentiable kernel (with $\ds \int_{\R}\rho_\theta =1$, and compact support $[-\theta,\theta]$), converging weakly to the dirac mass at $0$ in $\R$, in the sense of the distributions when $\theta \searrow 0$. We define for all $x\in  [0,+\infty)$:
\begin{eqnarray*}
\nonumber &\phi^{\varepsilon,\theta}(x)=\ds \int_{\R}\phi^{\varepsilon}(|z|)\rho_\varepsilon(x-z)dz.
\end{eqnarray*}
Since $\phi^{\varepsilon,\theta}$ does not depend on $i\in [I]$, the associated function defined on the junction $\mathcal{J}$ is then in the class $\mathcal{C}^2\big(\mathcal{J}\big)$. Fix now $\delta>0$ and $a>0$ two other parameters, where $a$ is large enough from $\varepsilon$, whereas $[\varepsilon-\delta,\varepsilon+\delta]\subset [0,a]$. To the parameter $a>0$, we associate the following stopping time:
$$\tau^a:=\inf\{~s\ge 0,~~x(s)\ge a~\},~~\P^{x_\star,i_\star}~~\text{a.s}.$$
Remark now that the regularity of $\phi^\varepsilon$ with the same arguments of proof used in Proposition 6.3 (Step 1) in \cite{Martinez-Ohavi Walsh}, lead to:
\begin{align*}
\lim_{\theta \searrow 0} \|\phi^{\varepsilon,\theta}(\cdot)-\phi^{\varepsilon}(\cdot)\|_{C^{1}([0,a])}=0,\\
\lim_{\theta \searrow 0} \|\phi^{\varepsilon,\theta}(\cdot)-\phi^{\varepsilon}(\cdot)\|_{C^{2}([0,\varepsilon-\delta]\cup[\varepsilon+\delta,a])}=0.
\end{align*}
On the other hand remark that:
\begin{align*}
&\limsup_{\delta \searrow 0}\E^{\P^{x_\star,i_\star}}\Big[~~\big|~~\ds \int_0^{\cdot}\frac{1}{2}\sigma_{i(s)}(s,x(s),l(s)\big(\phi^{\varepsilon,\theta}(x(s))-\phi^{\varepsilon}(x(s))\big)\ind{x(s)\in [\varepsilon-\delta,\varepsilon+\delta]}ds~~\big|_{(0,T)}\Big]\leq \\
&C\limsup_{\delta \searrow 0}\E^{\P^{x_\star,i_\star}}\Big[~~\ds \int_0^{T}\ind{x(s)\in [\varepsilon-\delta,\varepsilon+\delta]}ds\Big]=0, 
\end{align*}
where in the last equation $C>0$ is a standard constant, independent of $\delta$.
From It\^o's formula (established in Theorem \ref{th It\^o's formula} of this contribution), we have $\P^{x_\star,i_\star}$-a.s:
\begin{align*}
\begin{split}
&\Big (\phi^{\varepsilon,\theta}(x(s\wedge \tau^a))- \phi^{\varepsilon,\theta}(x_\star)=
\displaystyle\int_{0}^{s\wedge \tau^a}\Big(\displaystyle\frac{1}{2}\sigma_{i(u)}^2(u,x(u),l(u))\partial_x^2\phi^{\varepsilon,\theta}(x(u))\Big(\ind{x(s)\in [\varepsilon-\delta,\varepsilon+\delta]}+\\
&\mathbf{1}_{\{x(s)\in [0,\varepsilon-\delta]\cup[\varepsilon+\delta,a]\}}\Big)
+b_{i(u)}(u,x(u),l(u))\partial_x\phi^{\varepsilon,\theta}(x(u))\Big)du+\\
&\int_0^{s\wedge \tau^a}\sigma_{i(u)}(u,x(u),l(u))\partial_x\phi^{\varepsilon,\theta}(x(u))dW(u)+\displaystyle\sum_{i=1}^{I}\int_{0}^{^{s\wedge \tau^a}}\alpha_i(u,l(u))\partial_x\phi^{\varepsilon,\theta}(0)dl(u)\Big)_{0\leq s \leq T}.
\end{split}
\end{align*}
Hence the last arguments, will lead therefore if we send first $\delta \searrow 0$ (up to a sub sequence), and then after $\theta \searrow 0$, to obtain:
\begin{equation*}
\begin{split}
&\Big (\phi^{\varepsilon}(x(s\wedge \tau^a))-\phi^{\varepsilon}(x_\star)=\\
&\displaystyle\int_{0}^{s\wedge \tau^a}\displaystyle\frac{1}{2}\sigma_{i(u)}^2(u,x(u),l(u))\partial_x^2\phi^{\varepsilon}(x(u))
+b_{i(u)}(u,x(u),l(u))\partial_x\phi^{\varepsilon}(x(u))du\\
&\int_0^{s\wedge \tau^a}\sigma_{i(u)}(u,x(u),l(u))\partial_x\phi^{\varepsilon}(x(u))dW(u)+\displaystyle\sum_{i=1}^{I}\int_{0}^{^{s\wedge \tau^a}}\alpha_i(u,l(u))\partial_x\phi^{\varepsilon}(0)dl(u)\Big)_{0\leq s \leq T},\\
&\P^{x_\star,i_\star}~~\text{a.s}.
\end{split}
\end{equation*}
We conclude that the result stated at the beginning of this first \textbf{Step 1} holds true, using the monotone convergence, as soon as $a \nearrow +\infty$.

\textbf{Step 2:} Now we prove \eqref{eq : appro quadra sous}. Using the \textbf{Step 1}, the stochastic differential equation satisfied by the process $x(\cdot)$, and the expressions of the derivatives of $\phi^\varepsilon$ we obtain that:
\begin{eqnarray}\label{eq : decompo appro quad}
\nonumber \E^{\P^{x_\star,i_\star}}\Big[~~\Big|~~l(\cdot)-\sum_{j\in [I]}\frac{1}{2\varepsilon}\int_t^{\cdot}\sigma_j^2(s,x(u),l(s))\ind{0\leq x(s) \leq \varepsilon, i(s)=j}ds~~\Big|_{(0,T)}~~\Big] \\ \label{decomp 1}
\leq~~ C~~\Big(~~\E^{\P^{x_\star,i_\star}}\Big[~~\Big|\phi^\varepsilon(x(\cdot))- \phi^\varepsilon(x_\star)-\big(x(\cdot)-x_\star\big)~~\Big|_{(0,T)}~~\Big]\\\label{decomp 2}
+~~\E^{\P^{x_\star,i_\star}}\Big[~~\Big|~~\int_t^\cdot \sigma_{i(u)}(u,x(u),l(u))\ind{0\leq x(u) \leq \varepsilon}dW(u)~~\Big|_{(0,T)}~~\Big]\\ \label{decomp 3}
+~~\E^{\P^{x_\star,i_\star}}\Big[~~\Big|~~\int_t^\cdot b_{i(u)}(u,x(u),l(u))\ind{0\leq x(u) \leq \varepsilon}du~~\Big|_{(0,T)}~~\Big]~~
\Big),
\end{eqnarray}
for a strictly positive constant $C>0$ independent of $\varepsilon$.
We are going to prove that: \eqref{decomp 1}, \eqref{decomp 2}, \eqref{decomp 3} tend to $0$ as soon as $\varepsilon$ is sent to $\searrow 0$. For this purpose remark first that:
$$\forall x\ge 0,~~|\phi^\varepsilon(x)-x|\leq 2\varepsilon,$$
and then we obtain the required convergence for \eqref{decomp 1}. On the other hand, using assumption $(\mathcal{H})$ and the Burkholder-Davis-Gundy inequality, we get that there exists a constant $K>0$ independent of $\varepsilon$, such that:
$$\E^{\P^{x_\star,i_\star}}\Big[~~\Big|~~\int_0^\cdot \sigma_{i(u)}(u,x(u),l(u))\ind{\{0\leq x(s) \leq \varepsilon\}}dW(u)~~\Big|^2_{(0,T)}~~\Big]\leq K\E^{\P^{x_\star,i_\star}}\Big[~~\int_0^T\ind{0\leq x(u) \leq \varepsilon}du~~\Big].$$
Therefore, combining the {\it non-stickiness} property \eqref{eq : maj zero x} with the Cauchy-Schwarz inequality, \eqref{decomp 2} tends to $0$ as soon as $\varepsilon  \searrow 0$. Similarly we obtain that \eqref{decomp 3} tends to $0$ as soon as $\varepsilon  \searrow 0$. Finally, the uniform Lipschitz regularity of the coefficients $(\sigma_i)_{i\in [I]}$ with respect to their second variable: Assumption $(\mathcal{H})-({\bf R})~~(ii)$, implies:
$$\big|\sigma_{i(u)}(u,x(u),l(u))-\sigma_{i(u)}(u,0,l(u))\big|\ind{\{0\leq x(s) \leq \varepsilon\}}\leq \varepsilon|\sigma|\ind{\{0\leq x(s) \leq \varepsilon\}},~~\P_{t}^{x,i,l}~~\text{a.s.}$$
We obtain therefore using once again the {\it non-stickiness} condition \eqref{eq : maj zero x}
$$\lim_{\varepsilon \searrow 0} \E^{\P^{x_\star,i_\star}}\Big[~~\Big|~~\sum_{j\in [I]}\frac{1}{2\varepsilon}\int_0^{\cdot}\big(\sigma_{j}(u,x(u),l(u))-\sigma_{j}(u,0,l(u))\big)\ind{0\leq x(s) \leq \varepsilon, i(s)=j}ds~~\Big|_{(0,T)}~~\Big]=0,$$
and that completes the proof.
\end{proof}


\section{Problem 5 - Strong Markov property}
We discuss in this section on the strong Markov property related to the process involved in Theorem \ref{th: exis Spider}, and Corollary \ref{cr: exis Spider}.
\begin{Lemma}
\label{lem:null-set}
Let $\tau$ a $(\Psi_u)$-stopping time. Let $({\Q}^{Y,\tau}_{t})_{Y\in \Phi}$ a regular conditional distribution (r.c.p.d) of $\P_{t}^{x,i,l}(.|\Psi_\tau)$. For each $Y\in \Phi$ s.t. $\{\tau(Y) < \infty\}$, define $\hat{\Q}^{Y,\tau}_{t}$ on $(\Phi, \Psi)$ by 
\[
\hat{\Q}^{Y,\tau}_{t}(\omega \in A) = {\Q}_{t}^{Y,\tau}\pare{\omega(.+\tau)\in A},\;\;\;\;\forall A\in \Psi_{T}.
\]
Then there exists a $\P_{t}^{x,i,l} $ null set ${\mathcal N}\in \Psi_\tau$ such that for any $Y \notin {\mathcal N}\cup \{\tau = \infty\}$, the probability $\hat{\Q}^{Y,\tau}_{t}$ 
is a solution of $\big(\mathcal{S}_{pi}-\mathcal{M}_{ar}\big)$ starting from $Y(\tau)$ at time $t$.
\end{Lemma}
\begin{proof}

Denote $\Gamma :=\mathfrak{C}^{1,2,0}_{\{0\},b} \big(\mathcal{J}_T\times[0,+\infty)\big)$. There is a countable set $\Gamma^\ast$ dense in $\Gamma$ with respect to
bounded pointwise convergence of functions together with their first and second partial
derivatives (cf. \cite{Martinez-Ohavi Walsh}). By Theorem 1.2.10 of \cite{Stroock}, for each $f\in \Gamma^\ast$ there is a $\P_{t}^{x,i,l}$-null set ${\mathcal N}_f\in \Psi_\tau$, such that, for all $\omega: u\mapsto (x(u),i(u))\notin {\mathcal N}_f$,
\begin{align}
\chi^Y_f(s)(\omega) &:= f_{i(s)}(s,x(s),l(s))- f_{i_\star}(s\wedge \tau,x(s\wedge \tau),l(s\wedge \tau))\displaystyle\nonumber\\
&\;\;\;-\displaystyle \int_{s\wedge \tau}^{s}\sigma_{i(u)}(u,x(u),l(u))\partial_xf_{i(u)}(u,x(u),l(u))dW(u)\nonumber\\
&\;\;\;\displaystyle - \int_{t\wedge \tau}^{s}\partial_uf_{i(u)}(u,x(u),l(u))du-\int_{s\wedge \tau}^{s}b_{i(u)}(u,x(u),l(u))\partial_xf_{i(u)}(u,x(u),l(u))du\nonumber\\
&\;\;\;\displaystyle -\frac{1}{2}\int_{s\wedge \tau}^{s}\sigma_{i(u)}^2(u,x(u),l(u))\partial_{xx}^2f_{i(u)}(u,x(u),l(u))du\nonumber\\
&\;\;\;\displaystyle - \int_{s\wedge \tau}^{s}\Big[\partial_{l}f(u,0,l(u))dl(u) + \displaystyle\sum_{j=1}^I\alpha_j(u,l(u))\partial_{x}f_{j}(u,0,l(u))\Big]dl(u)
\end{align} 
defines a $\Q_{t}^{Y,\tau}$ martingale where $u\mapsto l(u)\in {\mathcal L}[\tau, T]$ satisfies that for each $u\in [\tau,T]$: 
\begin{eqnarray*}
\displaystyle\displaystyle\int_{\tau}^{u}\mathbf{1}_{\{x(\theta)>0\}}dl(\theta)=0,~~\Q_{t}^{Y,\tau}-\text{ a.s.}
\end{eqnarray*} Indeed, this comes from the fact that by assumption  $\Q_{t}^{Y,\tau}$ is a r.c.p.d. of $\P_{t}^{x_\star,i_\star,l_\star}\pare{\,.\,| \Psi_\tau}$ (note that such a r.c.p.d exists because $\Phi$, being a completely separable metric space, its Borel $\sigma$-field $\Psi$ is countably generated). 

Let ${\mathcal N} = \bigcup_{f\in \Gamma^\ast} {\mathcal N}_f$. Since the martingale property is preserved under bounded convergence for each time $t$, it follows by the density of $\Gamma^\ast$ in $\Gamma$ that $\chi^Y_f$ is a $\Q_{t}^{Y,\tau}$ martingale for each $Y\notin {\mathcal N}$ and $f\in \Gamma$. Then by Doob's submartingale stopping theorem, for
each positive integer $n$, 
\[
\left \{\chi^Y_f(s + \tau \wedge n), \Psi_{s + \tau \wedge n}~:~T> s\geq 0\right \}
\] is a $\Q^{Y,\tau}_{t}$ martingale. Letting $n\rightarrow +\infty$, it follows in view of continuity and bounded convergence that
$\left \{\ind{\tau<\infty}\chi^Y_f(s + \tau),\;\Psi_{s + \tau}~:~T> s\geq 0\right \}$ is a $\Q^{Y,\tau}_{t}$ martingale. Then, from the definition of $\hat{\Q}^{Y,\tau}_{t}$ in the statement of the lemma, it is easy to verify
that $\hat{\Q}^{Y,\tau}_{t}$ is a solution of the martingale problem $\big(\mathcal{S}_{pi}-\mathcal{M}_{ar}\big)$ starting from $Y(\tau)$ at time $t$, whenever $Y\notin {\mathcal N}\cup \{\tau = \infty\}$.
\end{proof}

\begin{Lemma}\label{lem: Markov}
The family $\big \{\P_{t}^{x,i,l} : t\in [0,T), (x,i)\in \mathcal{J}, l\in [0,+\infty) \big \}$ is strong Markov in the sense of Stroock-Varadhan for probability measures : \\
if $\tau >s $ is a finite stopping time then $\pare{\Pi_Y\otimes_{\tau(Y)}\P_{t=\tau(Y)}^{Y}}_{Y\in \Phi}$ is a r.c.p.d. of $\P_{t}^{x,i,l}\pare{\,.\,| \Psi_\tau}$.
\end{Lemma}
\begin{proof}
In the sequel, we denote $X(Y)=(x(Y),i(Y),l(Y))$, for any $Y$ belonging to the canonical space $\Phi$. We know that if $(\Q^{Y,\tau}_{t})_{Y\in \Phi}$ is a r.c.p.d. of $\P_{t}^{x,i,l}\pare{\,.\,| \Psi_\tau}$, then \[\Pi_{X(Y)} \otimes_{\tau(Y)} \Q^{Y,\tau}_{t} = \P_{t=\tau(Y)}^{X(Y)}\] for all $Y$ outside a $\P_{t}^{x_,i,l}$-null set $\mathcal{N}\in \Psi_\tau$. Thus,
\begin{align*}
\Q^{Y,\tau}_{t} &= \Pi_Y \otimes_{\tau(Y)}\pare{\Pi_{X(Y)} \otimes_{\tau(Y)} \Q^{Y,\tau}_{t}}\\
&=\Pi_Y \otimes_{\tau(Y)}\P_{t=\tau(Y)}^{X(Y)}\\
&=\Pi_Y \otimes_{\tau(Y)}\P_{t=\tau(Y)}^{x(\tau),i(\tau),l(\tau)}
\end{align*}
where the last equality follows from the uniqueness of the solution of the martingale problem \emph{$\big(\mathcal{S}_{pi}-\mathcal{M}_{ar}\big)$} and which holds for all $Y$ outside $\mathcal{N}$.
\end{proof}

\begin{Proposition}\label{pr: Markov fort}
The family $\{\P_{t}^{x,i,l} : t\in [0,T), (x,i,l)\in {\mathcal J}\times [0,+\infty)\}$  has the strong Markov property {\it i.e.} for any $(\Psi_u)$-stopping time $\tau$, each $(x,i,l)\in {\mathcal J}\times [0,+\infty)$ and any bounded measurable function $h\in {\mathcal C}({\mathcal J}\times [0,+\infty);\R)$, we have for any $u>s$:
\begin{align}
\label{eq:markov}
\E^{\P_{t}^{x,i,l} }\croc{\ind{\tau < \infty}h(x(u+\tau),i(u+\tau),l(u+\tau))| \Psi_\tau} = \ind{\tau <\infty}\E^{\P_{t=\tau}^{x(\tau),i(\tau),l(\tau)} }\croc{h(x(u),i(u),l(u))}
\end{align}
holding $\P_{t}^{x,i,l}$-a.s.
\end{Proposition}
\begin{proof}
We make use of the notations introduced in the previous Lemma \ref{lem:null-set}. 
From the result of Lemma \ref{lem: Markov}, we have
\begin{align}
\label{eq:rcpd}
&\nonumber \E^{\P_{t}^{x,i,l} }\croc{\ind{\tau < \infty}h(x(u+\tau),i(u+\tau),l(u+\tau))| \Psi_\tau}\\
&\nonumber =\Q^{Y,\tau}_{t}\croc{\ind{\tau < \infty}h(x(u+\tau),i(u+\tau),l(u+\tau))}\nonumber\\
&=\hat{\Q}^{Y,\tau}_{t}\croc{\ind{\tau < \infty}h(x(u),i(u),l(u))}
\end{align}
which holds whenever $Y\notin {\mathcal N}\cup \{\tau = \infty\}$

From the result of Lemma \ref{lem:null-set} and from the uniqueness of the martingale problem $\big(\mathcal{S}_{pi}-\mathcal{M}_{ar}\big)$ we can use $\P_{t=\tau(Y)}^{Y(\tau)}$ for $\hat{\Q}^{Y,\tau}_{t}$ in \eqref{eq:rcpd}. But then, \eqref{eq:markov} follows directly since $$Y(\tau)= ((x(\tau(Y)), i(\tau(Y))),l(\tau(Y))) = (x(\tau), i(\tau), l(\tau)),$$ is ensured to hold for $\P_{t}^{x,i,l}$--a.s. every $Y\in \Phi$.
\end{proof}

\section{Problem 6 - On the instantaneous scattering distribution along some ray $\mathcal{R}_i$}\label{sec scatering property}
In this section, we give a characterization of the scattering distribution along some ray $\mathcal{R}_i$, as soon as the spider $(x,i)$ reaches the junction point $\bf 0$.  

When the spinning measure is constant, namely:
$$\forall (t,l)\in[0,T]\times[0+\infty),~~\alpha_i(t,l)=\alpha_i,$$
and the coefficients of diffusion are homogeneous: $b_i(t,x,l)=b_i(x),~\sigma_i(t,x,l)=\sigma_i(x)$, it was proved in the seminal work \cite{freidlinS} (see Corollary 2.4), that for any $\delta>0$ (small enough), if we introduce the following stopping time:
$$\theta^\delta:=\inf\big\{~s\ge 0,~~x(s)=\delta~\big\},$$
then:
\begin{eqnarray}\label{eq distr instant diffra}
\forall i\in [I],~~\lim_{\delta \searrow 0} \P_{\bf 0}\big(~i(\theta^\delta)=i~\big) =\alpha_i.  
\end{eqnarray}
The last convergence shows that as soon as the homogeneous spider process $(x,i)$ reaches the junction point $\bf 0$, the 'instantaneous' probability distribution for the process $(x,i)$ to be scattered along the ray $\mathcal{R}_i$ is exactly equal to $\alpha_i$. 

We will see that an analogous result remains available for the solution of the martingale problem $\big(\mathcal{S}_{pi}-\mathcal{M}_{ar}\big)$. More precisely, we have the following proposition:
\begin{Proposition}\label{pr: scattering with level l}
Assume assumption $(\mathcal{H})$. Let $t\in [0,T)$ and $\ell\in [0,+\infty)$. Assume that $(x,i)=\bf 0$. Let $\P_t^{\bf 0,\ell}$ be the solution of the $\big(\mathcal{S}_{pi}-\mathcal{M}_{ar}\big)$ martingale problem given in Corollary \ref{cr: exis Spider}, starting at time $t$ from the junction point $(x,i)=\bf 0$, with a local-time level equal to $\ell$. Then for any $\delta>0$, if we introduce the following stopping time:
$$\theta^\delta:=\inf\big\{~s\ge 0,~~x(s)=\delta~\big\},$$
we have:
\begin{eqnarray}\label{eq distr instant diffra temp local}
\forall i\in [I],~~\lim_{\delta \searrow 0} \P_t^{\bf 0,\ell}\big(~i(\theta^\delta)=i~\big) =\alpha_i(t,\ell).  
\end{eqnarray}
\end{Proposition}
\begin{proof}
It follows first from the martingale property applied to $(x,i)\mapsto x$ that:
\begin{eqnarray}\label{ eq p1}
\delta=\E^{\P_t^{\bf 0,\ell}}\big[\ds \int_{t}^{\theta^\delta}b_{i(s)}(s,x(s),l(s))ds\big]+\E^{\P_t^{\bf 0,\ell}}\big[l(\theta^\delta)-\ell\big].
\end{eqnarray}
On the other hand, using the ellipticity condition given in assumption $(\mathcal{H})-({\bf E})$), and the map $(x,i)\mapsto x^2$, we get also that:
\begin{eqnarray}\label{ eq p2}
\nonumber & \delta^2-2\E^{\P_t^{\bf 0,\ell}}\big[\ds \int_{t}^{\theta^\delta}b_{i(s)}(s,x(s),l(s))x(s)ds\big]=\E^{\P_t^{\bf 0,\ell}}\big[\ds \int_{t}^{\theta^\delta}\sigma^2_{i(s)}(s,x(s),l(s))ds\big]\\
&\ge \underline{\sigma}^2 \E^{\P_t^{\bf 0,\ell}}[\theta^\delta-t].
\end{eqnarray}
Fix now $i\in [I]$ and define the following map $f$ by:
\begin{eqnarray*}
f:=\begin{cases}
    \mathcal{J}\to \R.\\
    (x,j)\mapsto x\mathbf{1}_{\{x\in \mathcal{R}_i\}}
\end{cases}.
\end{eqnarray*}
Once again, the martingale property applied to $f$ leads to:
\begin{eqnarray}\label{ eq p3}
\nonumber & \delta\P_t^{\bf 0,\ell}\big(~i(\theta^\delta)=i~\big)=\E^{\P_t^{\bf 0,\ell}}[\delta \mathbf{1}_{\{x(\theta^\delta)\in \mathcal{R}_i\}}]=\E^{\P_t^{\bf 0,\ell}}\big[\ds \int_{t}^{\theta^\delta}b_{i(s)}(s,x(s),l(s))\mathbf{1}_{\{x(s)\in \mathcal{R}_i\}}ds\big]+\\
&\E^{\P_t^{\bf 0,\ell}}\big[\ds \int_{t}^{\theta^\delta}\alpha_i(s,l(s))dl(s)\big].
\end{eqnarray}
Hence we obtain:
\begin{eqnarray}\label{ eq p4}
\nonumber & \P_t^{\bf 0,\ell}\big(~i(\theta^\delta)=i~\big)=\ds\frac{1}{\delta}\Big(\E^{\P_t^{\bf 0,\ell}}\big[\ds \int_{t}^{\theta^\delta}b_{i(s)}(s,x(s),l(s))\mathbf{1}_{\{x(s)\in \mathcal{R}_i\}}ds\big]\Big)+\\
&\ds\frac{1}{\delta}\big(\E^{\P_t^{\bf 0,\ell}}\big[\ds \int_{t}^{\theta^\delta}(\alpha_i(s,l(s))-\alpha_i(t,\ell))dl(s)\big]\Big)+\alpha_i(t,\ell)\ds\frac{1}{\delta}\E^{\P_t^{\bf 0,\ell}}[l(\theta^\delta)-\ell].
\end{eqnarray}
The 'non-stickiness' estimate given in Proposition 5.2 in \cite{Martinez-Ohavi Walsh}, recalled in \eqref{eq : maj zero x} leads to obtain:
$$\big|~\E^{\P_t^{\bf 0,\ell}}\big[\int_{t}^{\theta^\delta}b_{i(s)}(s,x(s),l(s))x(s)ds\big]~\big|\leq M\delta^2,$$
for a uniform constant $M>0$. From \eqref{ eq p2}, we get then that there exists a uniform constant $C>0$, such that:
\begin{eqnarray}\label{estimee centrale}
C\delta^2\ge \E^{\P_t^{\bf 0,\ell}}[\theta^\delta-t]\ge 0.    
\end{eqnarray}
Therefore with the aid of \eqref{estimee centrale}, we obtain easily in \eqref{ eq p1} and the first line of \eqref{ eq p4} that:
$$\lim_{\delta \searrow 0}\frac{1}{\delta}\E^{\P_t^{\bf 0,\ell}}[l(\theta^\delta)-\ell]=1, ~~\lim_{\delta \searrow 0}\ds\frac{1}{\delta}\Big|~\E^{\P_t^{\bf 0,\ell}}\big[\ds \int_{t}^{\theta^\delta}b_{i(s)}(s,x(s),l(s))\mathbf{1}_{\{x(s)\in \mathcal{R}_i\}}ds\big]~\Big|=0.$$
The Lipschitz regularity of the diffraction terms $\alpha_i$ - assumption $(\mathcal{H})-({\bf R})$ implies that (where $C>0$ is a uniform constant):
$$\frac{1}{\delta}\big|\E^{\P_t^{\bf 0,\ell}}\big[\ds \int_{t}^{\theta^\delta}(\alpha_i(s,l(s))-\alpha_i(t,\ell))dl(s)\big]\Big|\leq \frac{C}{\delta}\big(\E^{\P_t^{\bf 0,\ell}}\big[(l(\theta^\delta)-\ell)^2+(l(\theta^\delta)-\ell)(\theta^\delta-t)\big]\big).$$
Conditioning
with respect to the random variable $(\theta^\delta-t)$ in the last expectation, we conclude that the last term converges to $0$ as soon as $\delta\searrow 0$, using once again \eqref{estimee centrale} and the results obtained for the modulus of continuity of the local time given in Lemma 4.4 of \cite{Martinez-Ohavi Walsh}. We obtain therefore the required result sending $\delta\searrow 0$ in \eqref{ eq p4}.
\end{proof}
The last Proposition \ref{pr: scattering with level l} allows us to state that as soon as the spider process $(x,i)$ reaches the junction point $\bf 0$ at time $t$, with a level of local time $\ell$, the 'instantaneous' probability distribution for $(x,i)$ to be scattered along the ray $\mathcal{R}_i$ is then equal to the corresponding spinning coefficient $\alpha_i(t,\ell)$.

\medskip

\end{document}